\documentclass{amsart}
\usepackage{amsmath}
\usepackage{amsfonts}
\usepackage{amsthm}
\usepackage{amssymb}
\usepackage{tikz}
\usetikzlibrary{decorations.pathreplacing}
\usetikzlibrary{cd}
\usetikzlibrary{positioning}

\usepackage{algorithmic}
\usepackage{float}

\calclayout

\usepackage[backref]{hyperref}
\hypersetup{
  colorlinks   = true,          %Colors links instead of ugly boxes
  urlcolor     = blue,          %Color for external hyperlinks
  linkcolor    = purple,          %Color of internal links
  citecolor   = violet             %Color of citations
}

\newtheorem{theorem}{Theorem}
\newtheorem{lemma}{Lemma}

\theoremstyle{definition}
\newtheorem{definition}{Definition}
\numberwithin{definition}{subsection}
\newtheorem{remark}{Remark}
\numberwithin{remark}{subsection}
\newtheorem{question}{Question}
\numberwithin{question}{subsection}

\numberwithin{figure}{subsection}
\numberwithin{equation}{subsection}

\newcommand{\grid}[1]{
\foreach \x in {0,...,#1} {
	\pgfmathtruncatemacro\yend{#1-\x}
    \foreach \y in {0,...,\yend} {
    	\fill (\x,\y) circle (2pt);
    }
    }
    }
\newcommand{\sgrid}[2]{
\foreach \x in {0,...,#1} {
    \foreach \y in {0,...,#2} {
    	\fill (\x,\y) circle (2pt);
    }
}
}
\newcommand{\tropline}[5]{
\draw (#1,#2)--(#1+#3,#2+#3);
\draw (#1,#2)--(#1-#5,#2);
\draw (#1,#2)--(#1,#2-#4);
}

\usepackage[margin=0.95in]{geometry}
\usepackage{mathrsfs}

\title{Incidence geometry and universality in the tropical plane}
\author[]{Milo Brandt \ \ Michelle Jones \ \ Catherine Lee \ \ Dhruv Ranganathan}
\date{}

\begin{document}
\begin{abstract}
We examine the incidence geometry of lines in the tropical plane. We prove tropical analogs of the Sylvester--Gallai and Motzkin--Rabin theorems in classical incidence geometry. This study leads naturally to a discussion of the realizability of incidence data of tropical lines. Drawing inspiration from the von Staudt constructions and Mn\"ev's universality theorem, we prove that determining whether a given tropical linear incidence datum is realizable by a tropical line arrangement requires solving an arbitrary linear programming problem over the integers.
\end{abstract}
\maketitle

\renewcommand*{\thetheorem}{\Alph{theorem}}

\section{Introduction}

This paper investigates the incidence geometry of lines and points in the standard tropical plane, taking inspiration from fundamental theorems in combinatorial geometry. Our first results establish tropical versions of two classical theorems in the incidence geometry of $\mathbb R^2$: the Sylvester--Gallai theorem and the Motzkin--Rabin theorem. 

A set of points $\mathcal P\subset \mathbb R^2$ in the tropical plane \textbf{determines an ordinary tropical line} if there exists a tropical line in $\mathbb R^2$ passing through exactly two points in $\mathcal P$. 

\begin{theorem}[Tropical Sylvester--Gallai]\label{thm: tropical-SG}
Any set of four or more points in the tropical plane determines at least one ordinary tropical line.
\end{theorem}

A set of points $\mathcal P \subset \mathbb R^2$ with each point coloured either red or blue \textbf{determines a monochromatic tropical line} if there exists a tropical line in $\mathbb R^2$ containing at least two points of $\mathcal P$ and only containing points of the same colour.

\begin{theorem}[Tropical Motzkin--Rabin]\label{thm: tropical-MR}
Let $\mathcal P$ be a set of four or more points, each coloured red or blue, that determines finitely many tropical lines. Then $\mathcal P$ determines at least one monochromatic tropical line.
\end{theorem}

\noindent
We note that the conclusion of Theorem~\ref{thm: tropical-MR} fails without the supposition that the collection of points $\mathcal P$ determines finitely many lines. For instance, four points on the $x$-axis with alternating colours do not determine a monochromatic tropical line.

We prove the results above by analyzing the Newton subdivision of a dilated simplex that is induced the arrangement of tropical lines projectively dual to $\mathcal P$. Classically, a natural successor to the Sylvester--Gallai problem is the determination of the minimum number of ordinary Euclidean lines in an arrangement of given size. Our attempts to address this via Newton subdivisions leads to the following question: \textit{when is a polyhedral subdivision of a dilated standard simplex realized by an arrangement of tropical lines?} One might view this question as being a version of determining when a combinatorial geometry of lines is realized by an honest line arrangement in the projective plane. By a famous result of Mn\"ev, this latter realization problem is essentially unconstrained: the realization spaces of rank $3$ matroids can be arbitrary algebraic sets. 

We establish a piecewise linear analog of Mn\"ev's universality principle, drawing as inspiration von Staudt's ``algebra of throws''. We define a class of Newton subdivisions of the simplex, \textbf{linear Newton subdivisions}, which are candidates for subdivisions arising from line arrangements. We prove a universality theorem concerning such subdivisions.

\begin{theorem}[Tropical universality]\label{thm: universality}
For any polyhedral subset $S$ of $\mathbb R_{>0}^m$ defined by linear equalities and inequalities with integer coefficients, there exists a linear Newton subdivision whose realization space is linearly isomorphic to $S$.
\label{universalitythm}
\end{theorem}

\subsection{Combinatorial incidence geometry} The Sylvester--Gallai theorem states that any collection of non-collinear points in $\mathbb R^2$ determines a line passing through exactly two of the points. This theorem has its origins in a problem posed by James Joseph Sylvester in 1893. Tibor Gallai gave an elegant proof of the statement in 1943. It is believed that Sylvester's interest in the question derived from the classical geometry of plane curves. If $C\subset \mathbb P_{\mathbb C}^2$ is a smooth plane cubic, an explicit calculation shows that $C$ has precisely $9$ inflection points. Each line passing through two of these points contains a third. In other words, the Sylvester--Gallai theorem is false over the complex numbers. It follows that the $9$ inflection points of a smooth cubic defined over $\mathbb R$ cannot all have real coordinates. In fact, the result is also false over the field of $p$-adic numbers with $p\geq 5$, see~\cite{MO17}. 

Lines in the tropical plane may be viewed as limits of amoebas of curves in $(\mathbb C^\star)^2$, or as degenerations of families of lines in a two-dimensional algebraic torus over a valuation ring. As a result, the geometry of tropical curves in $\mathbb R^2$ frequently resembles projective geometry over $\mathbb C$. For instance, tropical curves in $\mathbb R^2$ satisfy an analogue of Bezout's theorem and, famously, may be used to count holomorphic curves in the projective plane. At the outset, it is unclear whether or not the Sylvester--Gallai property should hold in the tropical plane. Nonetheless, Theorem~\ref{thm: tropical-SG} shows that lines in the standard model of the tropical plane satisfy a Sylvester--Gallai theorem. We bring the reader's attention to two features of the result. First, the statement of tropical Sylvester--Gallai is stronger than the classical statement: as long as there are at least $4$ points in the collection, there is no need for the supposition that the points are non-collinear. If there are only $3$ points in the collection, then tropical non-collinearity forces the existence of an ordinary line. Second, the inflection points of a smooth tropical cubic are tropically collinear. This can be seen as a consequence of our result, or via direct computation, using the characterization of tropical inflections in~\cite{Bru11}. 

A variation on this theme, first posed by Graham and Newton, is as follows. In a collection of non-collinear points in the plane coloured red and blue, \textit{does there exist a monochromatic line passing through at least two points in the collection?} Motzkin and Rabin each provided a solution within a few years~\cite{chakerian_1970}. We refer the reader to the survey by Borwein and Moser for an overview of the history of the Sylvester--Gallai theorem and its generalizations~\cite{borwein_moser_1990}. On the tropical side, Theorem~\ref{thm: tropical-MR} establishes an analogous result. Again, however, the tropical results are not formal translations of the classical ones. A collection of coloured points that determines infinitely many lines may fail the Motzkin--Rabin property.

\subsection{Counting ordinary tropical lines and universality} Our route to the universality theorem passes through an enumerative version of the Sylvester--Gallai theorem. Specifically, we ask \textit{how many ordinary Euclidean lines can a collection of $n$ points in $\mathbb R^2$ determine?} In the Euclidean case, Green and Tao \cite{green_tao_2013} showed that for sufficiently large $n$, the minimum number of ordinary Euclidean lines through $n$ points is at least $n/2$, as conjectured by Dirac and Motzkin in 1951. 

We consider the analogous question for tropical lines and prove a strict lower bound on the minimum number of ordinary tropical lines determined by $n$ points. This lower bound is achieved by constructing appropriate subdivisions of the simplex, rather than the tropical arrangements themselves. The subdivisions we construct are not all realizable, which leads naturally to the question of when a subdivision of a lattice simplex into appropriate polytopes is realized by an arrangement of tropical lines. This realizability problem is the subject of our universality result. More generally, we hope that this paper is the beginning of a larger story involving the combinatorial incidence and enumerative geometry of tropical lines in $\mathbb R^2$, and take the universality theorem to be an indication of the potential richness of such a study. 

The universality statement presented in Theorem~\ref{thm: universality} takes inspiration from Mn\"ev's universality theorem. We recall that universality states that the realization space of a rank $3$ matroid can locally take the form of an arbitrary affine variety over the integers~\cite{mnev_1985,Mnev1988}. In other words, besides the obvious structure -- namely that they are constructible sets -- these realization spaces do not possess any special structure. Our result recovers this statement in the piecewise linear world. In algebraic geometry, one of the most spectacular applications of universality is in the law of Vakil--Murphy \cite{murph2004}. In this setting, universality is leveraged to show that many moduli spaces, notably of smooth curves in projective space, have unconstrained local geometry. We believe that a piecewise linear universality phenomenon should be ubiquitous in tropical moduli spaces and view the instance proved here as a starting point to establish such statements.

\subsection{Open questions}
We leave the following questions for future investigation.

\begin{question}
What is the minimum number of ordinary tropical lines determined by $n$ points?
\end{question}
\noindent In Section \ref{bnds}, we show that $n$ points determine at least $n-3$ ordinary tropical lines. However, this bound is not sharp; we are not aware of any family of arrangements determining fewer than $\frac{n^2}{4}$ points asymptotically.

\begin{question}
Do the tropical Sylvester--Gallai and Motzkin--Rabin theorems hold for lines in tropical linear spaces of rank $3$?
\end{question}
\noindent Our paper only considers these theorems in the standard tropical plane $\mathbb R^2$. The natural next case would be the tropicalization of the diagonal plane in $\mathbb P^3$. The central methods of this paper, namely Newton subdivisions associated to curves, do not immediately generalize to this case and new combinatorial ideas are most likely necessary. 

Another question, motivated by non-archimedean geometry, is the following. 
\begin{question}
Does there exist a linear embedding $\mathbb P^2\hookrightarrow \mathbb P^r$ for some $r$, such that the tropical Sylvester--Gallai theorem fails for tropical lines in the tropicalization of $\mathbb P^2$ in this embedding?
\end{question}

By Payne's limit theorem for Berkovich spaces, there exists some embedding of $\mathbb P^2$ in $\mathbb P^r$ such that the inflection points of a smooth elliptic curve tropicalize to distinct points~\cite{Pay09}. However, it is not immediate that this embedding can be taken to be linear.

\subsection{Outline} In Section~\ref{sec: background}, we introduce the basic objects of study -- tropical line arrangements and the subdivisions of simplices that they induce. The tropical Sylvester--Gallai theorem is proved in Section~\ref{tsg}. This section also contains a discussion for bounds on the number of ordinary tropical lines determined by a collection of points. Section~\ref{sec: motzkin-rabin} is dedicated to the proof of the tropical Motzkin--Rabin theorem, which follows from a careful analysis of edge-colourings in certain Newton subdivisions of the simplex. Finally, our universality result is proved in Section~\ref{universalitySection} by building incidence geometric descriptions of the basic operations composing a semilinear system of equations.

\subsection*{Acknowledgements} This research was completed as part of the 2017 Summer Undergraduate Mathematics Research Program at Yale (S.U.M.R.Y.) program. We are grateful to the mentors and participants for creating a stimulating research environment. The project has benefited from conversations with Derek Boyer, Melody Chan, Dave Jensen, Max Kutler, Hannah Markwig, Diane Maclagan, Andre Moura, Sam Payne, Ben Smith, Ayush Tewari, and Scott Weady. Finally, we thank both referees for their careful reading and helpful comments.

\numberwithin{theorem}{subsection}
\numberwithin{corollary}{subsection}
\numberwithin{lemma}{subsection}

\section{Background}\label{sec: background}
Throughout our paper we will restrict our discussion of tropical line arrangements to $\mathbb R^2$.

\begin{definition}
A \textbf{tropical line} is the corner locus of a convex piecewise linear function of the form $g(x,y)=\max\{x-a, y-b, 0\}$ for $(a,b)\in \mathbb R^2$. We call $(a,b)$ the \textbf{center} of the tropical line.
\end{definition}

Tropical lines are determined completely by their centers and have rays pointing in the directions of the \textbf{tropical basis vectors} $e_1=(1,1)$, $e_2=(-1,0)$, and $e_3=(0,-1)$. We call these rays the \textbf{axes} of a tropical line, and we say that two points are \textbf{coaxial} if there is a tropical line containing both points on the same axis, as in Figure \ref{a line}.

\begin{figure}[h!]
\begin{minipage}{0.495\textwidth}
\centering
\begin{tikzpicture}
\draw[->] (0,0) -- (1,1) node[pos=1,xshift=5pt,yshift=5pt]{$e_1$};
\draw[->] (0,0) -- (-1,0) node[pos=1,xshift=-7pt]{$e_2$};
\draw[->] (0,0) -- (0,-1) node[pos=1,yshift=-7pt]{$e_3$};
\fill (0,0) circle (2pt);
\end{tikzpicture}
\caption{A tropical line with its axes labeled by the basis vectors.}
\label{a line}
\end{minipage}
\hfill
\begin{minipage}{0.495\textwidth}
\centering
\begin{tikzpicture}
\draw[->] (0,0) -- (1,1);
\draw[->] (0,0) -- (-1.5,0);
\draw[->] (0,0) -- (0,-1.5);
\draw (-0.5,-0.5) -- (0,0);
\draw [->] (-0.5,-0.5) -- (-1.5,-0.5);
\draw [->] (-0.5,-0.5) -- (-0.5,-1.5);
\fill (0.3,0.3) circle (2pt);
\fill (.7,.7) circle (2pt);
\end{tikzpicture}
\vspace{0.3cm}
\caption{Two lines determined by a pair of coaxial points.}
\label{coaxial points}
\end{minipage}
\end{figure}

Two distinct points in the tropical plane determine exactly one tropical line unless they are coaxial, in which case they determine infinitely many tropical lines. Likewise, two distinct tropical lines intersect at exactly one point if their centers are not coaxial. Otherwise, the lines share an axis and intersect at infinitely many points.

\subsection{Tropical point-line duality}
As in projective geometry, there is a duality of the tropical plane switching the roles of points and lines while preserving incidence properties. Let $\mathscr A$ be the set of points and lines in the tropical plane. Projective duality is a map
\[
\phi: \mathscr A\to \mathscr A,
\]
taking a point $(a,b)$ to the tropical line centered at $(-a,-b)$, and taking the line with center at $(c,d)$ to the point $(-c,-d)$. The map is a containment reversing involution, i.e. $p\in \ell$ if and only if $\phi(\ell)\in \phi(p)$. Using this duality, we translate questions about arrangements of points into equivalent questions about arrangements of lines.

\subsection{Newton subdivisions}
Fix a convex piecewise linear function $F:\mathbb R^2\rightarrow\mathbb R$ defined as
\[F(x,y)=\max\{f(x,y):f\in \mathcal F\}\]
where $\mathcal F$ is a finite set of (affine) linear functions $\mathbb R^2\rightarrow \mathbb R$. The locus of $\mathbb R^2$ on which $F(x,y)$ is non-differentiable defines the \textbf{tropical curve of $F$}. This tropical curve is a graph, equipped with a piecewise linear embedding in $\mathbb R^2$. The data of this embedding can be recorded in a subdivision of an appropriate lattice polytope, as we now recall. We direct the reader to~\cite[Chapter 3]{MS15} for additional details. 

For each point $(x,y)\in \mathbb R^2$, consider the set of functions $f\in \mathcal F$ that achieve the maximum value, i.e. functions such that $f(x,y) = F(x,y)$. Define a polytope $\Delta_{(x,y)}$ to be the convex hull of the gradients of these functions. The Newton polygon $\Delta_F$ associated to $F$ is defined as the convex hull of gradients of \textit{all} functions $f\in \mathcal F$. Since the gradients used in forming $\Delta_{(x,y)}$ form a subset of those used in forming $\Delta_F$, there is a containment
\[
\Delta_{(x,y)}\subset \Delta_F.
\]
One can check that the polytopes $\Delta_{(x,y)}$ assemble to form the cells in a polyhedral subdivision of $\Delta_F$. We define this to be the \textbf{Newton subdivision} of $\Delta_F$ induced by $F$. %For the sake of brevity, when speaking of the \textbf{Newton subdivision attached to $F$}, we will always mean a Newton subdivision of the Newton polygon $\Delta_F$ attached to $F$.

\subsection{Newton subdivisions and line arrangements} Every collection of tropical lines in $\mathbb R^2$ defines a Newton subdivision of a dilated standard simplex. We now explain this procedure, and then identify a class of subdivisions of the simplex that could plausibly arise in this manner. 

\begin{definition}
An \textbf{arrangement of tropical lines} is a finite collection of tropical lines in $\mathbb R^2$. 
\end{definition}

Fix an arrangement of tropical lines $\mathcal L$. For each $\ell\in \mathcal L$, let $f_{\ell}(x,y)=\max\{x-a_{\ell},x-b_{\ell},0\}$ where $(a_{\ell},b_{\ell})$ is the center of $\ell$. The locus in $\mathbb R^2$ where this piecewise linear function is non-differentiable is precisely the tropical line $\ell$.  From the arrangement, we build a new convex piecewise linear function 
\[F=\sum_{\ell\in \mathcal L}f_{\ell}.\]
The corner locus of $F$ is the union of the tropical lines in the arrangement. We define the \textbf{Newton subdivision associated to a line arrangement $\mathcal L$} to be the Newton subdivision induced by the corresponding function $F$. Intuition can be gained from the classical analogue. The vanishing of a linear polynomial in $2$ variables in $\mathbb C^2$ defines a line, while the vanishing of a product of linear polynomials defines a line arrangement. Here, each piecewise linear function $f_\ell$ define a tropical line, and their ``tropical product'', i.e. their sum, defines a tropical arrangement.\\

\noindent
{\bf Explicit description of the Newton subdivision.} Using the procedure described in the preceding subsection, we build a subdivision of $\Delta_n$ induced by an arrangement of tropical lines. For arrangements of lines, the procedure can be made more explicit. Denote by $\Delta_n$ be the standard lattice triangle with side (lattice) length $n$. That is
\[
\Delta_n = \{(x,y):x\geq 0,\,y\geq 0,\,x+y\leq n\}
\] 
Let $\bar e_1=(-1,1)$ and $\bar e_2=(0,-1)$, and $\bar e_3=(1,0)$. 
  Fix a point $p\in \mathbb R^2$. We may describe the arrangement locally around $p$ as follows. Let $c$ be the number of tropical lines in $\mathcal L$ centered at $p$. Let $w_1,\,w_2$ and $w_3$ be the number of tropical lines $\ell\in\mathcal L$ such that $p$ is on the axis of $\ell$ in the direction of $e_1,\,e_2$ and $e_3$ respectively. The cell in the Newton subdivision corresponding to the local picture at $p$ will be a Minkowski sum of $c$ triangles, $w_1$ segments in the direction $\bar e_1$, $w_2$ segments in the direction $\bar e_2$, and $w_3$ segments in the direction $\bar e_3$. Call this polytope $P_{c,w_1,w_2,w_3}$. Figure \ref{i haven't made this figure} illustrates $P_{c,w_1,w_2,w_3}$. \\
  
 \noindent 
{\bf Linear and semiuniform subdivisions.}  The description above of Newton subdivisions attached to tropical line arrangements may be axiomatized, yielding a class of Newton subdivisions of the simplex that could plausibly arise from such arrangements. More precisely, define a \textbf{linear Newton subdivision} to be a polyhedral subdivision of $\Delta_n$ into faces of the form $P_{c,w_1,w_2,w_3}$. Given an arrangement of tropical lines, one may inspect the possible local pictures near the intersection points of the lines. An examination of these local possibilities shows that the Newton subdivision of a line arrangement always yields a linear Newton subdivision.

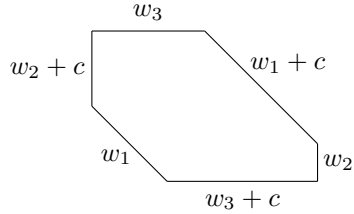
\begin{figure}
\begin{center}
\begin{tikzpicture} [scale = .5]
\draw (0,0) -- (0,2) node[midway,xshift=-17pt]{$w_2+c$};
\draw (0,2) -- (3,2) node[midway,yshift=7pt]{$w_3$};
\draw (3,2) -- (6,-1) node[midway,yshift=9pt,xshift=10pt]{$w_1+c$};
\draw (6,-1) -- (6,-2) node[midway,xshift=8pt]{$w_2$};
\draw (6,-2) -- (2,-2) node[midway,yshift=-7pt]{$w_3+c$};
\draw (2,-2) -- (0,0) node[midway,xshift=-5pt,yshift=-5pt]{$w_1$};
\end{tikzpicture}
\end{center}
\caption{The polygon $P_{c,w_1,w_2,w_3}$ labeled with its edge lengths.}
\label{i haven't made this figure}
\end{figure}

If $\mathcal L$ contains no pair of lines with coaxial centers, every edge in the Newton subdivision will have lattice length one. Figure \ref{Nine local cases} shows an enumeration of the possible faces appearing in the Newton subdivision of such a line arrangement. Define a linear Newton subdivision to be \textbf{semiuniform} if it only contains faces of the form $P_{c,0,0,0}$ and $P_{0,w_1,w_2,w_3}$ where $c,w_1,w_2,w_3\in \{0,1\}$. That is, the polytopes appearing in the subdivision are always Minkowski sums of $e_1$, $e_2$, and $e_3$.

\begin{figure} \begin{center}
\begin{tabular}{l l}
\begin{tikzpicture}[scale=0.85]
\draw (-1.25,0)--(0,0)--(0,-1.25); 
\draw (0,0) -- (.883,.883);
\fill[black] (0,0) circle (.07cm);
\draw[dashed] (0,0) circle (1.25cm);
\draw[->] (2,0) -- (3,0);
\draw (3.5,-.5) -- (3.5,.5) -- (4.5,-.5) -- cycle;
\fill[black] (3.5,-.5) circle (.07cm);
\fill[black] (3.5,.5) circle (.07cm);
\fill[black] (4.5,-.5) circle (.07cm);
\end{tikzpicture}
&
\begin{tikzpicture}[scale=0.85]
\fill[black] (0,0) circle (.07cm);
\draw[dashed] (0,0) circle (1.25cm);
\draw[->] (2,0) -- (3,0);
\fill[black] (3.5,0) circle (.07cm);
\end{tikzpicture}
\\
\begin{tikzpicture}[scale=0.85]
\draw (-1.25,0)--(1.25,0); 
\fill[black] (0,0) circle (.07cm);
\draw[dashed] (0,0) circle (1.25cm);
\draw[->] (2,0) -- (3,0);
\draw (3.5,-.5) -- (3.5,.5);
\fill[black] (3.5,-.5) circle (.07cm);
\fill[black] (3.5,.5) circle (.07cm);
\end{tikzpicture}
&
\begin{tikzpicture}[scale=0.85]
\draw (-.883,-.883)--(.883,.883); 
\fill[black] (0,0) circle (.07cm);
\draw[dashed] (0,0) circle (1.25cm);
\draw[->] (2,0) -- (3,0);
\draw (4,.5) -- (5,-.5);
\fill[black] (4,.5) circle (.07cm);
\fill[black] (5,-.5) circle (.07cm);
\end{tikzpicture}
\\
\begin{tikzpicture}[scale=0.85]
\draw (0,-1.25)--(0,1.25); 
\fill[black] (0,0) circle (.07cm);
\draw[dashed] (0,0) circle (1.25cm);
\draw[->] (2,0) -- (3,0);
\draw (3.5,0) -- (4.5,0);
\fill[black] (3.5,0) circle (.07cm);
\fill[black] (4.5,0) circle (.07cm);
\end{tikzpicture}
&
\begin{tikzpicture}[scale=0.85]
\draw (-.883,-.883) -- (.883,.883); 
\draw (0,-1.25) -- (0,1.25); 
\fill[black] (0,0) circle (.07cm);
\draw[dashed] (0,0) circle (1.25cm);
\draw[->] (2,0) -- (3,0);
\draw (3.5,0.5) -- (4.5,0.5) -- (5.5,-0.5) -- (4.5,-.5) -- cycle;
\fill[black] (3.5,.5) circle (.07cm);
\fill[black] (4.5,.5) circle (.07cm);
\fill[black] (5.5,-.5) circle (.07cm);
\fill[black] (4.5,-.5) circle (.07cm);
\end{tikzpicture}
\\
\begin{tikzpicture}[scale=0.85]
\draw (0,-1.25) -- (0,1.25); 
\draw (-1.25,0) -- (1.25,0); 
\fill[black] (0,0) circle (.07cm);
\draw[dashed] (0,0) circle (1.25cm);
\draw[->] (2,0) -- (3,0);
\draw (3.5,-.5) -- (3.5,0.5) -- (4.5,.5) -- (4.5,-.5) -- cycle;
\fill[black] (3.5,.5) circle (.07cm);
\fill[black] (3.5,-.5) circle (.07cm);
\fill[black] (4.5,.5) circle (.07cm);
\fill[black] (4.5,-.5) circle (.07cm);
\end{tikzpicture}
&
\begin{tikzpicture}[scale=0.85]
\draw (-.883,-.883) -- (.883,.883); 
\draw (-1.25,0) -- (1.25,0); 
\fill[black] (0,0) circle (.07cm);
\draw[dashed] (0,0) circle (1.25cm);
\draw[->] (2,0) -- (3,0);
\draw (3.5,0) -- (3.5,1) -- (4.5,0) -- (4.5,-1) -- cycle;
\fill[black] (3.5,0) circle (.07cm);
\fill[black] (3.5,1) circle (.07cm);
\fill[black] (4.5,0) circle (.07cm);
\fill[black] (4.5,-1) circle (.07cm);
\end{tikzpicture}
\\
\begin{tikzpicture}[scale=0.85]
\draw (-.883,-.883) -- (.883,.883); 
\draw (0,-1.25) -- (0,1.25); 
\draw (-1.25,0) -- (1.25,0); 
\fill[black] (0,0) circle (.07cm);
\draw[dashed] (0,0) circle (1.25cm);
\draw[->] (2,0) -- (3,0);
\draw (3.5,1) -- (4.5,1) -- (5.5,0) -- (5.5,-1) -- (4.5,-1) -- (3.5,0) -- cycle;
\fill[black] (3.5,1) circle (.07cm);
\fill[black] (4.5,1) circle (.07cm);
\fill[black] (5.5,0) circle (.07cm);
\fill[black] (5.5,-1) circle (.07cm);
\fill[black] (4.5,0) circle (.07cm);
\fill[black] (4.5,-1) circle (.07cm);
\fill[black] (3.5,0) circle (.07cm);
\end{tikzpicture}
\end{tabular}\caption{The nine possible local images of the corner locus of $F$ for arrangements without coaxial centers. Inside each dashed circle is a single point $p$ and the rays and lines through it, and to the right is the cell this local arrangement induces in the Newton subdivision.}\label{Nine local cases}
\end{center}\end{figure}
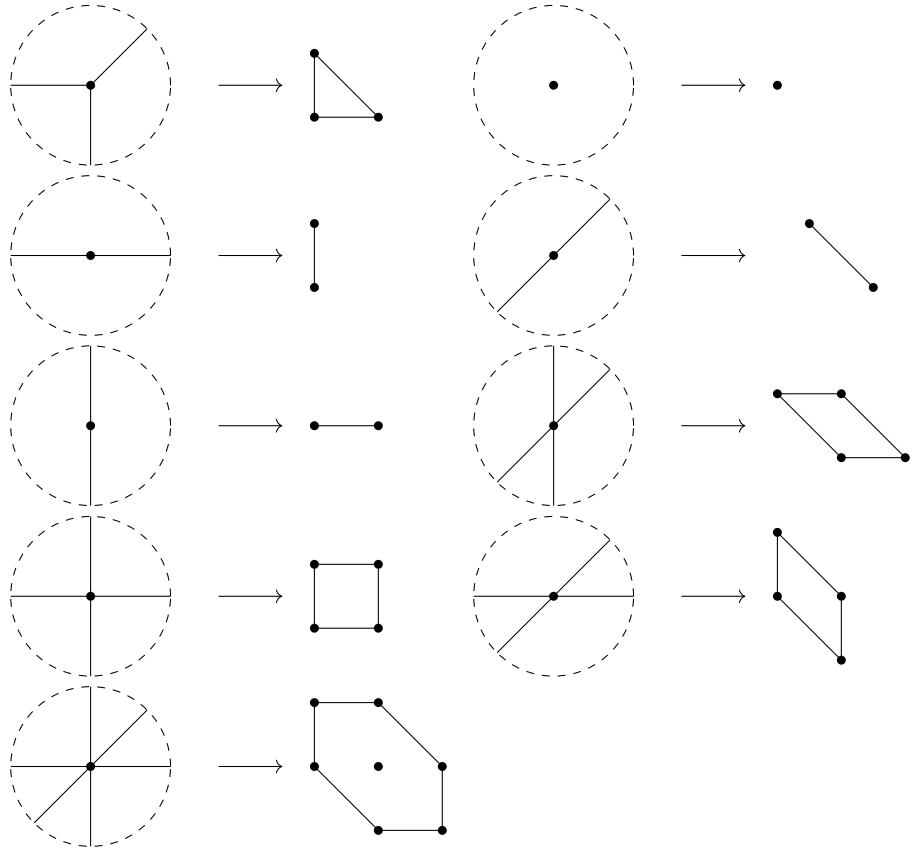

\begin{remark}
A linear Newton subdivision is meant to axiomatize the incidence data of a line arrangement, in the same manner that matroids axiomatize Euclidean line arrangements. The subdivision encodes the number of lines, which collections of lines meet, the axes along which the intersections lie, and the order of the intersection points along each axis of a given line. This incidence data can alternatively be recorded in an abstract combinatorial structure known as a \textbf{tropical oriented matroid}, defined and studied by Ardila and Develin~\cite{Ardila2009,Horn16,OhYoo}. The equivalence of these perspectives is established in their work. Our constructions of universal subdivisions are most natural from the perspective of linear Newton subdivisions, so we avoid the language of tropical oriented matroids. By the topological realization theorem for tropical oriented matroids, these may also be thought of as arrangements of tropical ``pseudolines'', namely the images of tropical lines under piecewise linear homeomorphisms~\cite{Horn16}. 
\end{remark}

\section{Tropical Sylvester--Gallai}
\label{tsg}

The purpose of this section is to establish Theorem~\ref{thm: tropical-SG}. This is split into two cases, dealing separately with arrangements of points containing a pair of coaxial points and those that do not.

\subsection{Coaxial case}

\begin{lemma}
Any set of points $\mathcal P$ in the tropical plane containing a pair of coaxial points determines an ordinary tropical line.
\end{lemma}

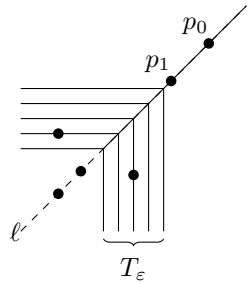
\begin{figure}[h!]
\begin{center}
\begin{tikzpicture}
\fill[black] (1,1) circle (.07cm) node[align=left,   above]{$p_0\quad$};
\fill[black] (0.5,0.5) circle (.07cm) node[align=left,   above]{$p_1\quad$};
\fill[black] (-.7,-.7) circle (.07cm);
\fill[black] (-1,-1) circle (.07cm);

\fill[black] (-1,-.2) circle (.07cm);
\fill[black] (0,-.75) circle (.07cm);

\draw[dashed] (-1.5,-1.5) -- (1.5,1.5) node[pos=0,xshift=-2pt]{$\ell$};

\draw (-1.5,-.4) -- (-.4,-.4) -- (1.5,1.5);
\draw (-.4,-1.5) -- (-.4,-.4);

\draw (-1.5,-0.2) -- (-0.2,-0.2);
\draw (-0.2,-1.5) -- (-0.2,-0.2);

\draw (-1.5,0) -- (0,0);
\draw (0,-1.5) -- (0,0);

\draw (-1.5,0.2) -- (0.2,0.2);
\draw (0.2,-1.5) -- (0.2,0.2);

\draw (-1.5,0.4) -- (0.4,0.4);
\draw (0.4,-1.5) -- (0.4,0.4);

\draw [decorate,decoration={brace,mirror},yshift = -.1cm] (-.4,-1.5) -- (.4,-1.5) node [black,midway,yshift = -.4cm] 
{$T_{\varepsilon}$};

\end{tikzpicture}
\end{center}
\caption{The coaxial case, showing several lines $T_{\varepsilon}$ as $\varepsilon$ varies.}
\label{coaxialProof}
\end{figure}

\begin{proof}
Consider the set of all lines determined by the point set $\mathcal P$. There is an infinite collection $\mathcal L_{12}$ of such lines passing through the coaxial points $p_1$ and $p_2$. Consider another point $p$. If $p$ is not coaxial with $p_1$ and $p_2$, then there is at most one line in $\mathcal L_{12}$ containing $p$. In this case, we see that an infinite collection of lines in $\mathcal L_{12}$ do not contain $p$. Since $p$ was an arbitrary non-coaxial point, there exists an infinite collection of lines in $\mathcal L_{12}$ that only contains points of $\mathcal P$ on a single axis. After a change of coordinates, we may assume that the coaxial points are $\{\lambda_i\cdot (1,1)\}_i$ where $\lambda_i$ is a finite set $S$ of positive real numbers. As illustrated in Figure~\ref{coaxialProof}, there is an ordinary tropical line passing through $\lambda_j\cdot(1,1)$ and $\lambda_k\cdot (1,1)$, where $\lambda_j$ and $\lambda_k$ are taken to be the two highest values appearing in the set $S$. This concludes the proof.
%Choose some Euclidean line $\ell$ parallel to some basis element $e_i$ containing more than one point in $\mathcal P$. Let us order the intersection $\ell \, \cap \, \mathcal P$ by saying $p \leq p'$ if there is an $\alpha\geq 0$ such that $p=p'-\alpha e_i$. This gives a total order. Let $p_0> p_1$ be the two largest vectors in this order. Figure \ref{coaxialProof} diagrams the situation. Consider the family of tropical lines $T_{\varepsilon}$ centered at $p_1-\varepsilon e_i$ for small non-negative $\varepsilon$. Any such tropical line passes through both $p_0$ and $p_1$. Moreover, the intersection of any two lines $T_{\varepsilon}$ and $T_{\varepsilon'}$ is precisely the Euclidean ray starting at $p_1-\min\{\varepsilon,\varepsilon'\}e_i$ in the direction of $e_i$. For small enough $\varepsilon$ and $\varepsilon'$, this ray will contain only $p_0$ and $p_1$ from $\mathcal P$. Define $\mathcal P'=\mathcal P\setminus \{p_0,p_1\}$. For small enough $\varepsilon$ and $\varepsilon'$ the intersection $(T_{\varepsilon}\cap \mathcal P') \cap (T_{\varepsilon'}\cap \mathcal P')$ is empty. Thus, the collection of all sets $T_{\varepsilon}\cap \mathcal P'$ is pairwise disjoint for sufficiently small $\varepsilon$. However, this collection is infinite and $\mathcal P'$ is finite, so there must be some $\varepsilon$ such that $T_{\varepsilon}\cap \mathcal P'$ is empty. Then $T_{\varepsilon}$ is an ordinary tropical line, as its intersection with $\mathcal P$ is precisely $\{p_0,p_1\}$.
\end{proof}

\subsection{Non-coaxial case}
A line arrangement $\mathcal L$ determines an \textbf{ordinary point} $p$ if there are exactly two tropical lines in $\mathcal L$ containing $p$. By point-line duality, the statement that an arrangement of points with no coaxial pairs determines an ordinary tropical line is equivalent to the following lemma.

\begin{lemma}
Let $\mathcal L$ be an arrangement of tropical lines such that no two centers of lines in $\mathcal L$ are coaxial. Then $\mathcal L$ determines an ordinary point.
\end{lemma}

\begin{proof}
Let $\mathcal L$ be an arrangement without coaxial centers. Let $N$ be the associated Newton subdivision of the simplex. The subdivision is clearly linear, and in fact semiuniform, since we assume that the centers are not coaxial. From Figure \ref{Nine local cases}, observe that ordinary points determined by $\mathcal L$ correspond to parallelograms in $N$. Thus, the condition that $\mathcal L$ determine no ordinary points is equivalent to the condition that $N$ is composed of only triangles and hexagons. We show that there is no subdivision of $\Delta_n$ for $n\geq 4$ containing only standard triangles of lattice side length $1$ and hexagons. 

Suppose there were such a subdivision. Then, there must be a triangle in the bottom left corner. The slanted edge of the triangle must be the base of a hexagon and two of the edges of this hexagon must be the base of triangles, as illustrated in Figure \ref{fiveshapes}. After adding these faces, there is a slanted edge composed of three diagonal segments. Each of these segments must be the base of a hexagon, but this is impossible because these hexagons would overlap. Thus no such subdivision exists.\end{proof}

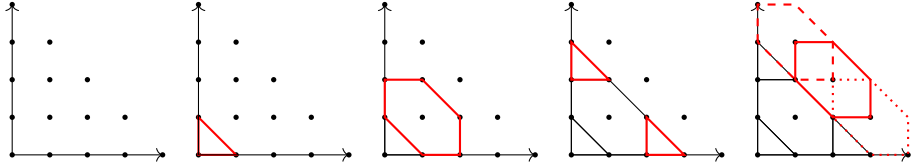
\begin{figure}[h!]
\begin{center}
\begin{tabular}{l c c c r}
\begin{tikzpicture} [scale=.5]
\draw[->] (0,0) -- (0,4);
\draw[->] (0,0) -- (4,0); 
\grid{4}
\end{tikzpicture}
&

\begin{tikzpicture} [scale = .5]
\grid{4}
\draw[red,thick] (0,0) -- (0,1) -- (1,0) -- cycle;
\draw[->] (0,0) -- (0,4);
\draw[->] (0,0) -- (4,0); 
\end{tikzpicture}
&
\begin{tikzpicture} [scale = .5]
\grid{4}
\draw (0,0) -- (0,1) -- (1,0) -- cycle;
\draw[->] (0,0) -- (0,4);
\draw[->] (0,0) -- (4,0); 
\draw[red,thick] (1,0) -- (0,1) -- (0,2) -- (1,2) -- (2, 1) -- (2,0) -- cycle;
\end{tikzpicture}
&

\begin{tikzpicture} [scale = .5]
\grid{4}
\draw (0,0) -- (0,1) -- (1,0) -- cycle;
\draw[->] (0,0) -- (0,4);
\draw[->] (0,0) -- (4,0); 
\draw (1,0) -- (0,1) -- (0,2) -- (1,2) -- (2, 1) -- (2,0) -- cycle;
\draw[red,thick] (0,2) -- (0,3) -- (1,2) -- cycle;
\draw[red,thick] (2,0) -- (2, 1) -- (3,0) -- cycle;
\end{tikzpicture}
&
\begin{tikzpicture} [scale = .5]
\grid{4}
\draw (0,0) -- (0,1) -- (1,0) -- cycle;
\draw[->] (0,0) -- (0,4);
\draw[->] (0,0) -- (4,0); 
\draw (1,0) -- (0,1) -- (0,2) -- (1,2) -- (2, 1) -- (2,0) -- cycle;
\draw (0,2) -- (0,3) -- (1,2) -- cycle;
\draw (2,0) -- (2, 1) -- (3,0) -- cycle;
\draw[red,thick,dashed] (0,3)--(0,4)--(1,4)--(2,3)--(2,2)--(1,2)--cycle;
\draw[red, thick] (1,2) -- (1,3)--(2,3)--(3,2)--(3,1)--(2,1)--cycle;
\draw[red,thick,dotted] (2,1)--(2,2)--(3,2)--(4,1)--(4,0)--(3,0)--cycle;
\end{tikzpicture}
\end{tabular}
\end{center}\caption{The faces forced in a subdivision avoiding parallelograms, ending with three overlapping hexagons.}
\label{fiveshapes}
\end{figure}
\subsection{Bounds for the number of ordinary tropical lines}
\label{bnds}
\begin{theorem}
Any semiuniform linear subdivision of $\Delta_n$ contains at least $n-3$ parallelograms if $n$ is a multiple of $3$ and at least $n-1$ otherwise. This bound is sharp.
\label{boundsandstuff}
\end{theorem}
\begin{proof} \textbf{The $n-3$ bound}. It suffices to consider the boundary edges of $\Delta_n$. Denote $E$ to be the set of $3n-3$ boundary edges remaining after removing each edge immediately counterclockwise of a vertex of the triangle, as shown in Figure \ref{orientedtriangle}.

\begin{figure}[h!]
\begin{minipage}{0.495\textwidth}
\centering
\begin{tikzpicture}
\draw[black,ultra thin] (0,0) -- (3,0) -- (0,3) -- cycle;
\draw[red,ultra thick] (1,0) -- (3,0);
\draw[red,ultra thick] (2,1) -- (0,3);
\draw[red,ultra thick] (0,2) -- (0,0);
\grid{3}
\end{tikzpicture}
\caption{The boundary of $\Delta_3$, with the edges in $E$ highlighted in red bold.}
\label{orientedtriangle}
\end{minipage}
\hfill
\begin{minipage}{0.495\textwidth}
\centering
\begin{tikzpicture}
\draw[red,very thick] (0,0) -- (1,0) node [black,midway,yshift=-.425cm] 
{$e$};
\draw[black,very thick] (1,0) -- (1,1) node[black,midway,xshift=.15cm,yshift=.1cm]{$d$};\draw[black] (1,1) -- (0,2) -- (-1,2) -- (-1,1) -- (0,0);
\draw[black] (1,0) -- (2,0) node [black,midway,yshift=-.42cm] 
{$e'$};
\draw[black] (2,0) -- (2,1) -- (1,1);
\draw[black] (2,0) -- (3,0) node [black,midway,yshift=-.42cm] 
{$e''$};
\draw[black] (3,0) -- (3,1) -- (2,1);
\draw[blue,very thick] (3,0) -- (4,0) node [black,midway,yshift=-.425cm] 
{$f(e)$};
\draw[black] (4,0) -- (3,1);

\draw[black,dashed,->] (0.5,-.25) -- (.5,.25) -- (3.5,.25) -- (3.5,-.25);
\draw[black,->] (0,-.65) -- (4,-.65);

\end{tikzpicture}
\caption{The construction of $f(e)$ from $e$.}
\label{injection}
\end{minipage}
\end{figure}

Each face of the subdivision contains at most one edge in $E$ and each edge in $E$ is contained in precisely one face. Define three subsets $E_H,E_T,E_P\subseteq E$, consisting of the edges contained in hexagons, triangles, and parallelograms respectively. Note $\vert E_T \vert \leq n$ because there are exactly $n$ triangles in any semiuniform subdivision of $\Delta_n$.

We describe an injection $f:E_H\rightarrow E_T$. Suppose that $e\in E_H$. Denote $e'$ to be the edge on the boundary immediately counterclockwise of $e$. Note that $e$ and $e'$ must be on the same side of $\Delta_n$. Let $d$ be the other edge of the hexagon containing the vertex $e'\cap e$, as shown in Figure \ref{injection}. Either $e'$ is the base of a parallelogram with edges parallel to $e$ and $d$, or it is the base of a triangle. If $e'$ is not the base of a triangle, then by the same argument $e''$ is either the base of a parallelogram or the base of a triangle. As the final edge on the side of the boundary cannot be the base of a parallelogram with edges parallel to $e$ and $d$, there must be some successor of $e$ that is the base of a triangle. Define $f(e)$ to be the first successor of $e$ that is the base of a triangle. Since all the edges between $e$ and $f(e)$ are in $E_P$, $f$ is injective.

As there is an injection $E_H\rightarrow E_T$ we have $\vert E_H \vert\leq \vert E_T \vert\leq n$. Since $\{E_H,E_T,E_P\}$ is a partition of $E$, we have $\vert E_H \vert + \vert E_T \vert + \vert E_P \vert = \vert E \vert = 3n-3$, which implies $\vert E_P \vert\geq n-3$. Any parallelogram can only include a single edge of $E$, so every subdivision must include at least $n-3$ parallelograms on its boundary.

\textbf{A mod $3$ refinement.} The area of $\Delta_n$ is $\frac{1}2n^2$, the area of a triangle is $\frac{1}2$, the area of a parallelogram is $1$, and the area of a hexagon is $3$. Therefore, if there are $h$ hexagons and $p$ parallelograms in a subdivision of $\Delta_n$, we must have $3h+p+\frac{1}2n = \frac{1}2n^2$. Rearranging and taking the equation modulo $3$ yields $p\equiv \frac{1}2(n^2-n)\pmod{3}$. Therefore $p\equiv 0\pmod{3}$ if $n\equiv 0$ or $1\pmod{3}$ and $p\equiv 1\pmod{3}$ otherwise. If $n$ is not divisible by $3$, this condition combined with the $n-3$ bound implies that there are at least $n-1$ parallelograms in the subdivision.

\textbf{Sharpness of the bound}. We can obtain this bound by filling in as much of the triangle as possible with a hexagonal subdivision beginning with a hexagon centered at $(1,1)$. Figure \ref{optimaltilings} illustrates the optimal subdivision for some small $n$. More formally, consider placing a hexagon centered at the coordinates $(a+2b,a-b)$ for every $a,b\in \mathbb Z$ such that this hexagon fits entirely within $\Delta_n$. 

Every point a distance of at least $1$ from the boundary will be within the union of these hexagons. Moreover, every space left uncovered by these hexagons will be filled by parallelograms and triangles. It is clear that such a subdivision has $n-1$ parallelograms if $n\equiv 1$ or $2\pmod{3}$ and $n-3$ parallelograms otherwise.

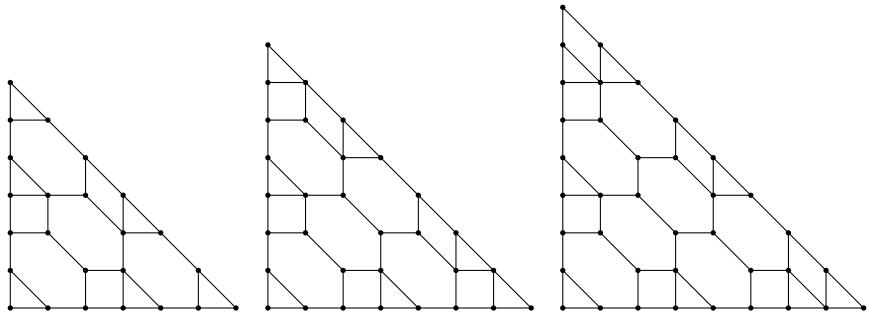
\begin{figure}[h!]
\begin{center}
\begin{tabular}{ccc}
\begin{tikzpicture}[scale=.5]

\grid{6}
\draw[black] (0,0) -- (0,6) -- (6,0) -- cycle;
\draw[black] (1,0) -- (0,1);

\draw[black] (2,0) -- (2,1) -- (1,2) -- (0,2);
\draw[black] (1,2) -- (1,3) -- (2,3) -- (3,2) -- (3,1) -- (2,1);
\draw[black] (1,3) -- (0,4);
\draw[black] (2,3) -- (2,4);
\draw[black] (0,5) -- (1,5);

\draw[black] (3,1) -- (4,0);
\draw[black] (3,2) -- (4,2);
\draw[black] (5,0) -- (5,1);

\draw[black] (0,3) -- (1,3);
\draw[black] (3,0) -- (3,1);
\draw[black] (3,2) -- (3,3);

\fill[white] (1,1) circle (3pt);
\fill[white] (2,2) circle (3pt);
\fill[white] (4,1) circle (3pt);
\fill[white] (1,4) circle (3pt);
\end{tikzpicture} &
\begin{tikzpicture}[scale=.5]

\grid{7}
\draw[black] (0,0) -- (0,7) -- (7,0) -- cycle;
\draw[black] (1,0) -- (0,1);

\draw[black] (2,0) -- (2,1) -- (1,2) -- (0,2);
\draw[black] (1,2) -- (1,3) -- (2,3) -- (3,2) -- (3,1) -- (2,1);
\draw[black] (1,3) -- (0,4);
\draw[black] (2,3) -- (2,4);
\draw[black] (0,5) -- (1,5);

\draw[black] (1,6) -- (1,5) -- (2,4) -- (3,4);

\draw[black] (6,1) -- (5,1) -- (4,2) -- (4,3);

\draw[black] (3,1) -- (4,0);
\draw[black] (3,2) -- (4,2);
\draw[black] (5,0) -- (5,1);

\draw[black] (0,3) -- (1,3);
\draw[black] (3,0) -- (3,1);

\draw[black] (0,6) -- (1,6);
\draw[black] (2,4) -- (2,5);
\draw[black] (6,0) -- (6,1);
\draw[black] (5,1) -- (5,2);

\fill[white] (1,1) circle (3pt);
\fill[white] (2,2) circle (3pt);
\fill[white] (3,3) circle (3pt);
\fill[white] (4,1) circle (3pt);
\fill[white] (1,4) circle (3pt);
\end{tikzpicture} &
\begin{tikzpicture}[scale=.5]

\grid{8}
\draw[black] (0,0) -- (0,8) -- (8,0) -- cycle;
\draw[black] (1,0) -- (0,1);

\draw[black] (2,0) -- (2,1) -- (1,2) -- (0,2);
\draw[black] (1,2) -- (1,3) -- (2,3) -- (3,2) -- (3,1) -- (2,1);
\draw[black] (1,3) -- (0,4);
\draw[black] (2,3) -- (2,4);
\draw[black] (0,5) -- (1,5);
\draw[black] (1,6) -- (2,6);
\draw[black] (6,1) -- (6,2);
\draw[black] (3,5) -- (3,4) -- (4,3) -- (5,3);

\draw[black] (1,6) -- (1,5) -- (2,4) -- (3,4);

\draw[black] (6,1) -- (5,1) -- (4,2) -- (4,3);

\draw[black] (3,1) -- (4,0);
\draw[black] (3,2) -- (4,2);
\draw[black] (5,0) -- (5,1);

\draw[black] (0,3) -- (1,3);
\draw[black] (3,0) -- (3,1);

\draw[black] (0,6) -- (1,6);
\draw[black] (6,0) -- (6,1);

\draw[black] (1,6) -- (0,7);
\draw[black] (1,6) -- (1,7);
\draw[black] (4,3) -- (4,4);
\draw[black] (7,0) -- (7,1);
\draw[black] (7,0) -- (6,1);

\fill[white] (1,1) circle (3pt);
\fill[white] (2,2) circle (3pt);
\fill[white] (3,3) circle (3pt);
\fill[white] (4,1) circle (3pt);
\fill[white] (5,2) circle (3pt);
\fill[white] (1,4) circle (3pt);
\fill[white] (2,5) circle (3pt);

\end{tikzpicture}
\end{tabular}
\end{center}
\caption{Subdivision achieving the optimal bound on parallelograms for $n=6,7,8$.}
\label{optimaltilings}
\end{figure}
\end{proof}
\begin{remark}
Theorem \ref{boundsandstuff} gives a lower bound for the number of ordinary points determined by an arrangement of $n$ tropical lines. However, the bound is not sharp because for $n=8$ or $n\geq 10$, our optimal subdivision method does not produce a subdivision corresponding to any line arrangement. The question of which subdivisions are realizable is further explored in Section \ref{universalitySection}. Note however, that if one asks the question for arrangements of tropical pseudolines, i.e. the images of tropical lines under piecewise linear homeomorphisms, the arrangements described above are optimal and can all be realized~\cite{Horn16}.
\end{remark}

\section{Tropical Motzkin--Rabin}\label{sec: motzkin-rabin}

We now consider the tropical Motzkin--Rabin theorem, which concerns two-coloured arrangements of tropical lines in $\mathbb R^2$. That is, an arrangement of lines where each line is assigned a colour, which we will take to be either red or blue. A coloured arrangement of lines $\mathcal L$ determines a \textbf{monochromatic point} $p$ if $p$ is contained by at least two lines of $\mathcal L$ and all lines containing $p$ in $\mathcal L$ are the same colour. By tropical projective duality, Theorem \ref{thm: tropical-MR} is equivalent to the following lemma.

\begin{lemma} \label{monochromatic point-line}
Any arrangement of red and blue tropical lines $\mathcal L$ in the tropical plane with no two centers coaxial determines a monochromatic point.
\end{lemma}

\noindent For the reader's convenience, throughout this section we will represent red lines with dashed lines and blue lines with solid lines.

\subsection{Coloured  subdivisions}
We extend the method of Newton subdivisions to coloured line arrangements. Let $\mathcal L$ be an arrangement of red and blue lines with Newton polytope $\Delta$ and induced Newton subdivision $N$. Recall from Figure \ref{Nine local cases} that the edges of $N$ are in correspondence with segments of axes of tropical lines. We can lift the colouring of $\mathcal L$ to a colouring of the edges of $N$. Any subdivision arising in this manner will be linear, and satisfy the following two properties.

\begin{enumerate}
\item All edges of a triangle within a Newton division are the same colour, as they correspond to segments of axes of the same tropical line.
\item Any pair of parallel edges of a parallelogram or hexagon are the same colour, as these correspond to segments of the same axis of a tropical line (see Figure \ref{Lifting}).
\end{enumerate}

We refer to any linear subdivision of the simplex with a colouring of the edges that satisfy the rules above as a \textbf{plausibly coloured subdivision}. 

\begin{figure}
\begin{center}
\begin{tikzpicture}[scale=0.85]
\draw[red,very thick,dashed] (-.883,-.883) -- (.883,.883); 
\draw (0,-1.25) -- (0,1.25); 
\draw (-1.25,0) -- (1.25,0); 
\fill[black] (0,0) circle (.07cm);
\draw[dashed] (0,0) circle (1.25cm);
\draw[->] (2,0) -- (3,0);
\draw[black] (3.5,1) -- (4.5,1);
\draw[red,very thick,dashed] (4.5,1) -- (5.5,0);
\draw[black] (5.5,0) -- (5.5,-1) -- (4.5,-1);
\draw[red,very thick,dashed] (4.5,-1) -- (3.5,0);
\draw[black] (3.5,0) -- (3.5,1);
\fill[black] (3.5,1) circle (.07cm);
\fill[black] (4.5,1) circle (.07cm);
\fill[black] (5.5,0) circle (.07cm);
\fill[black] (5.5,-1) circle (.07cm);
\fill[black] (4.5,0) circle (.07cm);
\fill[black] (4.5,-1) circle (.07cm);
\fill[black] (3.5,0) circle (.07cm);
\end{tikzpicture}
\caption{A particular instance of lifting the colour from a tropical line to a Newton subdivision.}
\label{Lifting}
\end{center}
\end{figure}
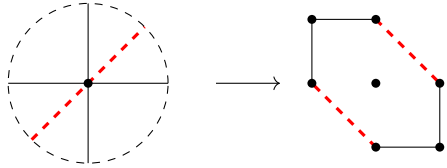

A \textbf{monochromatic polygon} is a polygon whose edges are all the same colour. If the Newton subdivision of $\mathcal L$ contains a monochromatic parallelogram or hexagon, then $\mathcal L$ determines a monochromatic point. In the following section we will prove the tropical Motzkin--Rabin theorem by showing that no semiuniform linear subdivision with side length at least $4$ avoids monochromatic hexagons and parallelograms.

\subsection{Forbidden patterns in counterexamples}
We first determine a class of patterns that may never appear in any plausibly coloured linear subdivision free of monochromatic hexagons and parallelograms. 

\begin{definition}
An \textbf{$(n,m)$-arm} is a chain of edges in a plausibly coloured subdivision starting with a red edge in the direction $-\bar{e_1}$, followed by $n$ red edges in the direction $\bar{e_3}$, then $m$ blue edges in the direction $\bar e_2$, and ending with a blue edge in the direction $-\bar{e_1}$.
\end{definition}

An $(n,m)$-arm is depicted in Figure \ref{Armsarmsarms}. This figure also illustrates the only ways in which faces may be placed above the arms while avoiding monochromatic hexagons and parallelograms. To see that these exhaust all the cases, consider building a subdivision extending the arm. Inspecting the leftmost edge, we note that it is slanted and adjacent to a horizontal edge of the same colour. Thus, only a hexagon or a parallelogram may be placed here. After this step, we are left with two edges of opposite colour intersecting perpendicularly. In particular, a triangle cannot be placed here due to the colouring rules, and the only polytope that can be placed above this edge is a square. This argument may be repeated to the end of the arm. Similar reasoning applies to the bottom-most edge.

Thus, if any $(n,m)$-arm appears in a plausibly coloured  subdivision free of monochromatic hexagons and parallelograms, the colouring rules of the subdivision force an $(n',m')$ arm above and to the right of the first arm, where $n'\in \{n,n+1\}$ and $m'\in \{m,m+1\}$. Because each arm forces another arm above and to the right of it, it is not possible for there to be finitely many arms; hence no arms can appear in such a subdivision.

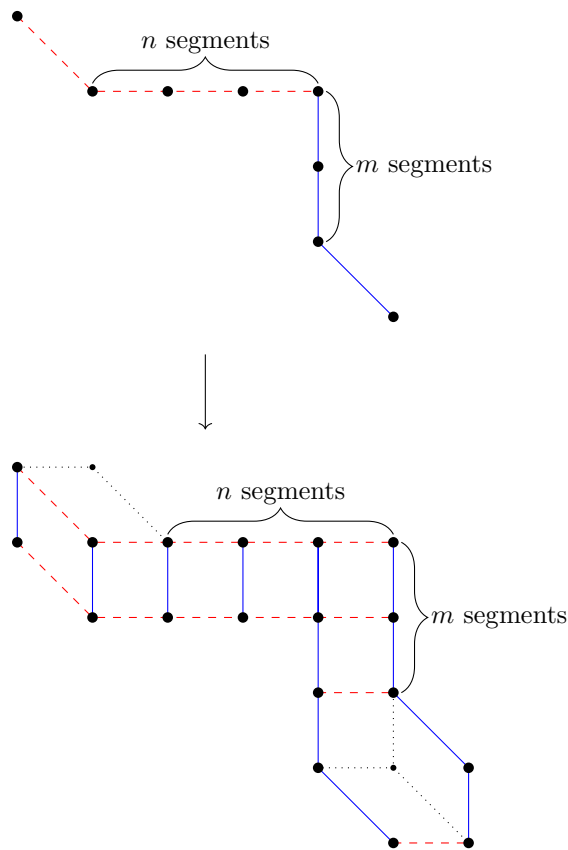
\begin{figure}
\begin{center}
\begin{tikzpicture}
\draw[red,dashed] (0,0) -- (1,-1) -- (2,-1) -- (3,-1) -- (4, -1);
\draw[blue] (4,-1) -- (4,-2) -- (4,-3) -- (5,-4);
\draw [decorate,decoration={brace,amplitude=10pt},yshift=3pt]
(1,-1) -- (4,-1) node [black,midway,yshift=0.55cm] 
{$n$ segments};
\draw [decorate,decoration={brace,amplitude=10pt,mirror},xshift=3pt]
(4,-3) -- (4,-1) node [black,midway,xshift=1.3cm] 
{$m$ segments};

\fill[black] (0,0) circle (.07cm);
\fill[black] (1,-1) circle (.07cm);
\fill[black] (2,-1) circle (.07cm);
\fill[black] (3,-1) circle (.07cm);
\fill[black] (4,-1) circle (.07cm);
\fill[black] (4,-2) circle (.07cm);
\fill[black] (4,-3) circle (.07cm);
\fill[black] (5,-4) circle (.07cm);

\draw[black,->] (2.5,-4.5) -- (2.5,-5.5);

\begin{scope}[shift={(0,-7)}]
\draw[red,dashed] (0,0) -- (1,-1) -- (2,-1) -- (3,-1) -- (4, -1);
\draw[blue] (4,-1) -- (4,-2) -- (4,-3) -- (5,-4);
\draw[red,dashed] (2,0) -- (3,0) -- (4,0);
\draw[blue] (5,0) -- (5,-1) -- (5,-2);
\draw[red,dashed] (4,0) -- (5,0);
\draw[red,dashed] (4,-1) -- (5,-1);
\draw[blue] (3,0) -- (3,-1);
\draw[blue] (4,0) -- (4,-1) -- (4,0);

\draw[blue] (0,0) -- (0,1);
\draw[blue] (2,-1) -- (2,0);
\draw[red,dashed] (0,1) -- (1,0) -- (2,0);
\draw[blue] (1,0) -- (1,-1);
\draw[red,dashed] (5,-4) -- (6,-4);
\draw[red,dashed] (4,-2) -- (5,-2);
\draw[blue] (6,-4) -- (6,-3) -- (5,-2);

\draw[black,dotted] (0,1) -- (1,1) -- (2,0);
\fill[black] (1,1) circle (.04cm);

\draw[black,dotted] (5,-2) -- (5,-3) -- (6,-4);
\draw[black,dotted] (5,-3) -- (4,-3);
\fill[black] (5,-3) circle (.04cm);

\fill[black] (0,0) circle (.07cm);
\fill[black] (1,-1) circle (.07cm);
\fill[black] (2,-1) circle (.07cm);
\fill[black] (3,-1) circle (.07cm);
\fill[black] (4,-1) circle (.07cm);
\fill[black] (4,-2) circle (.07cm);
\fill[black] (4,-3) circle (.07cm);
\fill[black] (5,-4) circle (.07cm);

\fill[black] (0,1) circle (.07cm);
\fill[black] (1,0) circle (.07cm);
\fill[black] (2,0) circle (.07cm);
\fill[black] (3,0) circle (.07cm);
\fill[black] (4,0) circle (.07cm);

\fill[black] (5,-2) circle (.07cm);
\fill[black] (6,-3) circle (.07cm);
\fill[black] (6,-4) circle (.07cm);
\fill[black] (5,-1) circle (.07cm);
\fill[black] (5,0) circle (.07cm);
\draw [decorate,decoration={brace,amplitude=10pt},yshift=3pt]
(2,0) -- (5,0) node [black,midway,yshift=0.55cm] 
{$n$ segments};
\draw [decorate,decoration={brace,amplitude=10pt,mirror},xshift=3pt]
(5,-2) -- (5,0) node [black,midway,xshift=1.3cm] 
{$m$ segments};
\end{scope}
\end{tikzpicture}
\end{center}
\caption{A $(n,m)$-arm shown with one of four ways to tile above an $(n,m)$ arm while avoiding monochromatic hexagons and parallelograms. The dotted black edges show edges present in the other three possible subdivisions. On each end of the figure, one may choose to place either a hexagon or two parallelograms above the outermost edges.
}
\label{Armsarmsarms}
\end{figure}
\subsubsection{The $(0,0)$-arm}
We see that a $(0,0)$-arm cannot appear in a plausibly coloured subdivision except on the boundary of $\Delta_n$, which is seen as follows. If it is not on the boundary, then there must be faces above the arm. Up to exchanging colours and reflectional symmetry, the only ways to place faces above a $(0,0)$-arm while avoiding monochromatic hexagons and parallelograms are shown in Figure \ref{badshapes}. Either way, there will necessarily be a $(1,0)$ or a $(1,1)$-arm, which are forbidden.

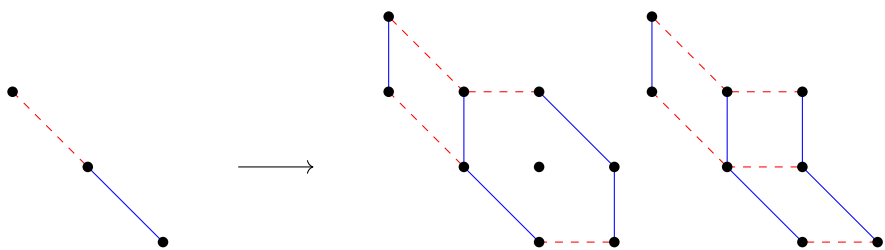
\begin{figure}
\begin{tikzpicture}
\draw[red, dashed] (0,0) -- (1,-1);
\draw[blue] (1,-1) -- (2,-2);

\fill[black] (0,0) circle (.07cm);
\fill[black] (1,-1) circle (.07cm);
\fill[black] (2,-2) circle (.07cm);

\draw[->] (3,-1) -- (4,-1);

\begin{scope}[shift={(5,0)}]
\draw[red, dashed] (0,0) -- (1,-1);
\draw[blue] (1,-1) -- (2,-2);
\draw[red, dashed] (2,-2) -- (3,-2);
\draw[blue] (3,-2) -- (3,-1);
\draw[blue] (3,-1) -- (2,0);
\draw[red, dashed] (2,0) -- (1,0);
\draw[blue] (1,0) -- (1,-1); 
\draw[red, dashed] (1,0) -- (0,1);
\draw[blue] (0,1) -- (0,0);

\fill[black] (0,0) circle (.07cm);
\fill[black] (1,-1) circle (.07cm);
\fill[black] (2,-2) circle (.07cm);
\fill[black] (3,-2) circle (.07cm);
\fill[black] (2,-1) circle (.07cm);
\fill[black] (3,-1) circle (.07cm);
\fill[black] (2,0) circle (.07cm);
\fill[black] (1,0) circle (.07cm);
\fill[black] (0,1) circle (.07cm);
\end{scope}
\begin{scope}[shift={(8.5,0)}]
\draw[red, dashed] (0,0) -- (1,-1);
\draw[blue] (1,-1) -- (2,-2);
\draw[blue] (1,-1) -- (1,0);
\draw[red, dashed] (1,0) -- (0,1);
\draw[blue] (0,1) -- (0,0);
\draw[red, dashed] (1,-1) -- (2,-1);
\draw[blue] (2,-1) -- (3,-2);
\draw[red, dashed] (3,-2) -- (2,-2);
\draw[blue] (2,-1) -- (2,0);
\draw[red, dashed] (2,0) -- (1,0);

\fill[black] (0,0) circle (.07cm);
\fill[black] (1,-1) circle (.07cm);
\fill[black] (2,-2) circle (.07cm);
\fill[black] (2,-1) circle (.07cm);
\fill[black] (1,0) circle (.07cm);
\fill[black] (0,1) circle (.07cm);
\fill[black] (3,-2) circle (.07cm);
\fill[black] (2,0) circle (.07cm);
\end{scope}
\end{tikzpicture}
\caption{A $(0,0)$-arm and the two options for placing non-monochromatic hexagons and parallelograms above it.}
\label{badshapes}
\end{figure}
\subsection{Subdividing from the corner}
To prove Theorem \ref{thm: tropical-MR}, we show that any subdivision of $\Delta_n$ for $n\geq 4$ without monochromatic hexagons or parallelograms must contain an arm. We examine the behavior of such a subdivision around its bottom left corner. Up to reflection, the three cases depicted in Figure \ref{Motzkincases} represent all of the possible shpes for the corner of a plausibly coloured  subdivision.

\begin{figure}
\begin{center}
\begin{tabular}{c c c}
\begin{tikzpicture}
\draw[black] (0,0) -- (1,0);
\draw[black] (1,0) -- (1,1);
\draw[black] (1,1) -- (0,1);
\draw[black] (0,1) -- (0,0);
\draw[black,->] (0,1) -- (0,3);
\draw[black, ->] (1,0) -- (3,0);
\end{tikzpicture}
&
\begin{tikzpicture}
\draw[black] (0,0) -- (1,0);
\draw[black] (1,0) -- (0,1);
\draw[black] (0,1) -- (0,0);
\draw[black] (2,0) -- (2,1) -- (1,2) -- (0,2);

\draw[black,->] (0,1) -- (0,3);
\draw[black,->] (1,0) -- (3,0);
\end{tikzpicture} 
&
\begin{tikzpicture}
\draw[black] (0,0) -- (1,0);
\draw[black] (1,0) -- (0,1);
\draw[black] (0,1) -- (0,0);
\draw[black] (1,0) -- (1,1) -- (0,2);

\draw[black,->] (0,1) -- (0,3);
\draw[black,->] (1,0) -- (3,0);
\end{tikzpicture}
\end{tabular}
\end{center}
\caption{An enumeration of the distinct cases arising in the lower left corner while constructing a plausibly coloured subdivision.}
\label{Motzkincases}
\end{figure}
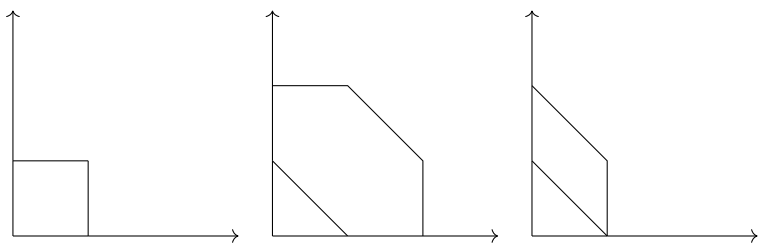
\subsubsection{The case of a square in the corner}
Suppose there is a square in the bottom left corner. This ensures that the left and bottom boundaries begin with opposite colours. Without loss of generality, assume that the left edge begins with a blue edge and the bottom edge begins with a red edge. Suppose that there are exactly $b$ consecutive blue edges on the left and $r$ consecutive red edges on the right. To avoid monochromatic hexagons and parallelograms, the colouring rules force a $r\times b$ grid of squares with triangles on both the left and bottom boundaries, as shown in Figure \ref{squarebad}. This forces an $(r-1,b-1)$-arm; thus this case can only yield a subdivision free of monochromatic hexagons and parallelograms if $r=b=1$ and $n=2$.

\subsubsection{The case of a hexagon on top of a triangle in the corner}
Suppose there is a hexagon on top of a triangle in the corner. Without loss of generality, assume that the diagonal edges of the hexagon are red and the vertical edges are blue; the colour of the horizontal edges is inconsequential. Let $r\geq 0$ be the number of consecutive red edges on the bottom boundary following the hexagon's horizontal edges. Each of these edges must be the base of a square and this chain of squares must followed by a blue triangle, as in Figure \ref{hexatoast}. Thus any subdivision starting with a triangle and hexagon must include an $(r,0)$-arm. If such a subdivision does not contain monochromatic hexagons or parallelograms, we must have $r=0$ and $n=3$.

\begin{figure}
\begin{minipage}{0.43\textwidth}
\centering
\begin{tikzpicture}
\sgrid{3}{2}
\draw[red, dashed] (0,0) -- (3,0);
\draw[red, dashed] (0,1) -- (3,1);
\draw[red, dashed] (0,2) -- (3,2);

\draw[red, dashed] (0,2) -- (0,3);
\draw[red, dashed] (0,3) -- (1,2);

\draw[blue] (0,0) -- (0,2);
\draw[blue] (1,0) -- (1,2);
\draw[blue] (2,0) -- (2,2);
\draw[blue] (3,0) -- (3,2);

\draw[blue] (3,0) -- (4,0);
\draw[blue] (4,0) -- (3,1);

\fill[black] (4,0) circle (.07cm);
\fill[black] (0,3) circle (.07cm);

\draw [decorate,decoration={brace,amplitude=10pt},xshift = -.1cm] (0,0) -- (0,2) node [black,midway,xshift = -1.2cm] {$b$ segments};
\draw [decorate,decoration={brace,amplitude=10pt,mirror},yshift = -.1cm] (0,0) -- (3,0) node [black,midway,yshift = -.5cm] 
{$r$ segments};
\end{tikzpicture}
\caption{The case of a square in the bottom left corner and the $(r-1,b-1)$ arm that must appear.}
\label{squarebad}
\end{minipage}
\hfill
\begin{minipage}{0.495\textwidth}
\centering
\begin{tikzpicture}
\draw[red, dashed] (0,0) -- (1,0);
\draw[red, dashed] (1,0) -- (0,1);
\draw[red, dashed] (0,1) -- (0,0);
\draw[black,dotted] (1,0) -- (2,0);
\draw[blue] (2,0) -- (2,1);
\draw[red, dashed] (2,1) -- (1,2);
\draw[black,dotted] (1,2) -- (0,2);
\draw[blue] (0,2) -- (0,1);

\draw[red, dashed] (2,1) -- (4,1);
\draw[red, dashed] (2,0) -- (4,0);
\draw[blue] (3,0) -- (3,1);
\draw[blue] (5,0) -- (4,0) -- (4,1);
\draw[blue] (4,1) -- (5,0);

\fill[black] (0,0) circle (.07cm);
\fill[black] (0,1) circle (.07cm);
\fill[black] (1,0) circle (.07cm);
\fill[black] (1,1) circle (.07cm);
\fill[black] (0,2) circle (.07cm);
\fill[black] (1,2) circle (.07cm);
\fill[black] (2,1) circle (.07cm);
\fill[black] (2,0) circle (.07cm);

\fill[black] (3,1) circle (.07cm);
\fill[black] (3,0) circle (.07cm);
\fill[black] (4,1) circle (.07cm);
\fill[black] (4,0) circle (.07cm);
\fill[black] (5,0) circle (.07cm);

\fill[white] (3,3) circle (.01cm);

\draw [decorate,decoration={brace,amplitude=10pt,mirror},yshift = -.1cm] (2,0) -- (4,0) node [black,midway,yshift = -.5cm] 
{$r$ segments};
\end{tikzpicture}
\caption{The case of a hexagon on top of a triangle in the bottom left corner and the $(r,0)$-arm that must appear.}
\label{hexatoast}
\end{minipage}
\end{figure}

\subsubsection{The case of a parallelogram on top of a triangle in the corner}
This case follows from the preceding one by contracting the horizontal edges of the hexagon. Thus the only subdivision that begins with this configuration and is free of monochromatic hexagons and parallelograms is a subdivision of $\Delta_2$.

\subsection{Counterexamples in the coaxial case.}
Theorem \ref{thm: tropical-MR} requires that no two points in the arrangement be coaxial. Without this restriction, we can find arrangements that do not determine monochromatic lines. One such arrangement is illustrated in Figure \ref{please don't be mad at my naming conventions}. This counterexample has the property that every unbounded edge in the line arrangement is the overlap of a red and a blue axis.

This counterexample can be used to construct further counterexamples. For any line arrangement $\mathcal L$, we may paste a copy of this counterexample into any empty region of $\mathcal L$. Doing so creates no new monochromatic points and eliminates any monochromatic points through which an unbounded edge of the counterexample passes. By adding enough copies of the counterexamples to any arrangement, one can obtain an arrangement that determines no monochromatic points.

\begin{figure}
\begin{center}
\begin{tikzpicture}
\draw[red, dashed] (-1,0) -- (-1,-1);
\draw[blue] (0,1) -- (0,0);
\draw[red, dashed] (1,1) -- (1,0);

\draw[red, dashed] (0,-1) -- (-1,-1);
\draw[blue] (0,0) -- (1,0);
\draw[red, dashed] (0,1) -- (1,1);

\draw[red, dashed] (-1,0) -- (0,1);
\draw[blue] (-1,-1) -- (0,0);
\draw[red, dashed] (0,-1) -- (1,0);

\draw[violet,->, dotted, thick] (0,0) -- (-2,0);
\draw[violet,->, dotted, thick] (0,1) -- (-1,1);
\draw[violet,->, dotted, thick] (-1,-1) -- (-2,-1);

\draw[violet,->, dotted, thick] (0,0) -- (2,2);
\draw[violet,->, dotted, thick] (0,1) -- (1,2);
\draw[violet,->, dotted, thick] (1,0) -- (2,1);

\draw[violet,->, dotted, thick] (0,0) -- (0,-2);
\draw[violet,->, dotted, thick] (-1,-1) -- (-1,-2);
\draw[violet,->, dotted, thick] (1,0) -- (1,-1);

\fill[red, dashed] (0,0) circle (2pt);
\fill[red, dashed] (-1,0) circle (2pt);
\fill[red, dashed] (0,-1) circle (2pt);
\fill[red, dashed] (1,1) circle (2pt);

\fill[blue] (0,1) circle (2pt);
\fill[blue] (1,0) circle (2pt);
\fill[blue] (-1,-1) circle (2pt);

\end{tikzpicture}
\end{center}
\caption{A counterexample to the tropical Motzkin-Rabin theorem in the coaxial case. The vertices of tropical lines are marked with circles and the rays of the tropical lines are drawn in the corresponding colour, with dotted purple axes representing where red and blue axes overlap.}
\label{please don't be mad at my naming conventions}
\end{figure}
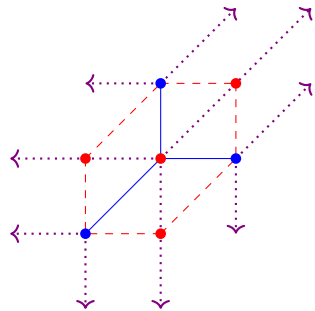
\section{Realizations of Newton subdivisions}
\label{universalitySection}
In this section, we will study realization spaces of Newton subdivisions. Informally, a realization space should be thought of as the space of tropical line arrangements with a given Newton subdivision.

Let $N$ be a linear Newton subdivision of a dilated standard triangle. Define the set of interior edges $E_N$ of a Newton subdivision $N$ to consist of the edges not contained in the boundary of $N$. These edges correspond to bounded segments in the corresponding tropical line arrangement. For any bounded segment in a tropical line arrangement, the displacement from one end of the segment to the other is $\pm \alpha e_i$ for some $i$; we call $\alpha$ the length of the segment. It is possible to reconstruct a line arrangement, up to translation, if we know its Newton subdivision $N$ and the length associated with each edge in $E_N$. 

Define the set of interior vertices $V_N$ to consist of the vertices not on the boundary of $N$. For any $v \in V_N$, define $E_v$ to be the set of edges containing $v$, and for any edge $e \in E_v$, define $\overrightarrow{e_v}$ to be the basis vector pointing along $e$ outwards from $v$.

\begin{definition}
An \textbf{exact metric} on a Newton subdivision $N$ is a map $d:E_N\rightarrow \mathbb R_{>0}$ such that for every vertex $v\in V_N$ we have \[\sum_{e\in E_v}d(e)\overrightarrow{e_v}=0\]
\end{definition}

The value of $d$ on an edge represents the corresponding length in the line arrangement. The exactness condition expresses the fact that the face vectors of any polytope must sum to zero. An example of an exact metric and a corresponding tropical line arrangement is shown in Figure \ref{exactnessness}. Figure \ref{milo claims to explain} illustrates the necessity of the exactness condition.

\begin{figure}
\begin{center}
\begin{tikzpicture}[scale=.75]
\grid{3}
\draw (0,0) -- (3,0) -- (0,3) -- cycle;

\draw[very thick] (0,1) -- (1,0) node[midway,yshift=-4pt,xshift=-2pt]{$1$};
\draw[very thick] (1,1) -- (1,0) node[midway,xshift=5pt]{$1$};
\draw[very thick] (1,1) -- (0,2) node[midway,yshift=-4pt,xshift=-2pt]{$1$};
\draw[very thick] (1,1) -- (2,1) node[midway,yshift=5pt,xshift=-2pt]{$1$};
\draw[very thick] (0,2) -- (1,2) node[midway,yshift=5pt,xshift=-2pt]{$2$};
\draw[very thick] (2,0) -- (2,1) node[midway,xshift=5pt,yshift=-2pt]{$3$};
\fill (1,1) circle (4pt);

\draw[->] (3,1.5)--(4,1.5);

\begin{scope}[shift={(6,1)},scale=.75]

\draw[very thick] (0,0) -- (1,0);
\draw[very thick] (1,0) -- (1,1);
\draw[very thick] (1,1) -- (0,0);
\draw[very thick] (0,0) -- (-1,-1);
\draw[very thick] (1,0) -- (4,0);
\draw[very thick] (1,1) -- (1,3);
\draw (1,0) -- (1,-2);
\draw (-1,-1) -- (-1,-2);
\draw (-1,-1) -- (-2,-1);
\draw (0,0) -- (-2,0);
\draw (4,0) -- (4,-2);
\draw (4,0) -- (5,1);
\draw (1,3) -- (-2,3);
\draw (1,3) -- (2,4);
\draw (1,1) -- (4,4);

\fill[gray,opacity=.25] (0,0) -- (1,0) -- (1,1) -- cycle;

\end{scope}

\end{tikzpicture}
\caption{An exact metric on a linear Newton subdivision and its corresponding tropical line arrangement. The edges of $E_N$ are bolded in $N$ and correspond to the bounded edges on the right. The interior vertex of $V_N$ is highlighted and corresponds to the shaded face of the arrangement.}
\label{exactnessness}
\end{center}
\end{figure}
\begin{figure}
\begin{tabular}{l r}

\begin{tikzpicture}[scale=.75]
\grid{3}
\draw (0,0)--(0,1)--(1,0) -- cycle;
\draw (0,1)--(0,2)--(1,1)--(1,0)--cycle;
\draw (1,0)--(1,1)--(2,1)--(2,0)--cycle;
\draw (2,0)--(2,1)--(3,0)--cycle;
\draw (0,2)--(1,2)--(2,1)--(1,1)--cycle;
\draw (0,2)--(0,3)--(1,2);
\draw[->] (3,2)--(4,2);
\node at (.5,1.2) {1};
\node at (1.2,.5) {1};
\node at (1.25, 1.3) {3};
\end{tikzpicture}
&
\begin{tikzpicture}
\draw[very thick, densely dotted, red] (-.35,-.65)--(0,-.65)--(0,.4)--(-.35,.05);
\tropline{0}{.7}{.5}{.3}{1};
\tropline{-1.2}{-.8}{.85}{.5}{.5};
\tropline{.5}{-.65}{1}{1}{.5};
%\draw[dashed] (-.35,-.65)--(-1,-.65);
\end{tikzpicture}
\end{tabular} \caption{An assignment of lengths not satisfying the exactness condition around the interior vertex. The failure of the exactness condition corresponds to the failure of the red edges on the right to form a polygon.} \label{milo claims to explain}
\end{figure}

We define the \textbf{realization space} $\mathcal R(N)$ of a linear Newton subdivision $N$ to be the space of exact metrics on $N$. The realization space can be identified with a polyhedral subset of $\mathbb R_{>0}^m$ on which certain integer linear equalities and inequalities hold. In this section, we work towards showing that every subspace defined by linear equalities and inequalities with integer coefficients is linearly isomorphic to $\mathcal R(N)$ for some subdivision $N$.

\subsection{Construction of universal subdivisions}
We now describe a way to construct a linear Newton subdivision whose realization space satisfies some particular set of equalities. Recalling that linear Newton subdivisions capture the combinatorial data of a line arrangement, our goal is to specify incidences that are satisfied exactly on a given subspace of tropical line arrangements.

The space of tropical line arrangements whose centers share a common vertical axis is identified with $\mathbb R_{>0}^m$. Choose coordinates on this space by defining the $i^{th}$ coordinate to be distance between the $i^{th}$ and $(i+1)^{th}$ horizontal line (see Figure \ref{encodingStep}). We will describe a construction, representable using a linear Newton subdivision, that enforces further constraints between the coordinates without introducing any degrees of freedom. As a pictorial outline of the proof, Figure \ref{nonnewtonianPicture} shows the line arrangement we use to force a single linear equality between the distances between lines. Figure \ref{complete} shows a series of these constraining constructions pasted together. Figure \ref{newtonianPicture} shows the Newton subdivision corresponding to a line arrangement forcing a single constraint. Lemma \ref{newtonConstruction} formalizes the way in which these constructions fit together and describes the process of adding a single constraint to a realization space.

Our construction may only ever force constraints of the form
\begin{equation}\sum_{a\leq j < b}v_j=\sum_{a'\leq j <b'}v_j,\label{constraint}\end{equation}
where the $v_j$ are the affine coordinates on the space of intersections of lines with the common vertical axis described above. 
Geometrically, this corresponds to forcing the distance between one pair of horizontal axes to equal the distance between some other pair of horizontal axes. The construction requires that the tuple $(a',b',a,b)$ satisfies $a<b$, $a'<b'$, $a'<a$ and $b'<b$. We call such tuples \textbf{admissible} and we say that a subspace $S$ of $\mathbb R_{>0}^m$ is \textbf{intervallic} if there is some finite set $I$ of admissible tuples such that $S$ is precisely the subspace on which Equation \ref{constraint} holds for all $(a',b',a,b)\in I$.

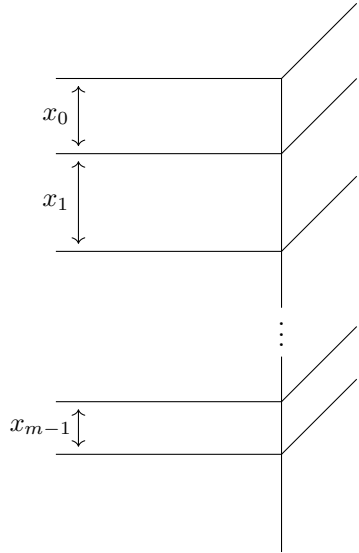
\begin{figure}
\begin{center}
\begin{tikzpicture}
\draw (-3,0) -- (0,0) -- (1,1);
\draw (-3,-1) -- (0,-1) -- (1,0);
\draw (-3,-2.3) -- (0,-2.3) -- (1,-1.3);

\node[black] at (0,-3.3){$\vdots$};
\draw (-3,-4.3) -- (0,-4.3) -- (1,-3.3);
\draw (-3,-5) -- (0,-5) -- (1,-4);
\draw (0,0) -- (0,-3.05);
\draw (0,-3.7) -- (0,-6.3);

\draw [<->] (-2.7,-0.1) -- (-2.7,-0.9) node [black,midway,xshift=-.3cm] {$x_0$};
\draw [<->] (-2.7,-1.1) -- (-2.7,-2.2) node [black,midway,xshift=-.3cm] {$x_1$};

\draw [<->] (-2.7,-4.4) -- (-2.7,-4.9) node [black,midway,xshift=-.5cm] {$x_{m-1}$};

\end{tikzpicture}
\end{center}
\caption{A Newton subdivision with realization space $\mathbb R^m$, with the natural coordinates marked.}
\label{encodingStep}
\end{figure}

The construction we use to add a single constraint to a realization space assumes that every unbounded horizontal edge of the arrangement intersects a common vertical edge of a tropical line. This is a rigidification that is necessary to ensure that adding more constraints will not interfere with previously constructed structure of the realization space.

Let us say that a Newton subdivision $N$ is \textbf{$m$-extensible} if the left edge of the subdivision consists of $m+2$ edges with vertices $v_{-1},\ldots,v_{m}$ such that each vertex $v_i$ for $0\leq i < m$ is connected to precisely one bounded edge $\ell_i$ in the direction $\bar e_3$. For any $m$-extensible Newton subdivision, define a map $\pi_N:\mathcal R(N)\rightarrow \mathbb R_{>0}^m$ to be the map taking a exact metric $d$ to the vector $(d(\ell_0),\ldots,d(\ell_{m-1}))$.

\begin{lemma}
For any $m$-extensible Newton subdivision $N$ and admissible tuple $(a',b',a,b)$, there exists an $m$-extensible Newton subdivision $N'$ and a injective linear map $\sigma:\mathcal R(N')\rightarrow \mathcal R(N)$ such that $\pi_{N'}=\pi_N\circ \sigma$ and whose image is precisely the subset of $d\in \mathcal R(N)$ such that
\[\sum_{a\leq i < b}\pi_N(d)_i=\sum_{a'\leq i < b'}\pi_N(d)_i.\]
\label{newtonConstruction}
\end{lemma}

\begin{proof}
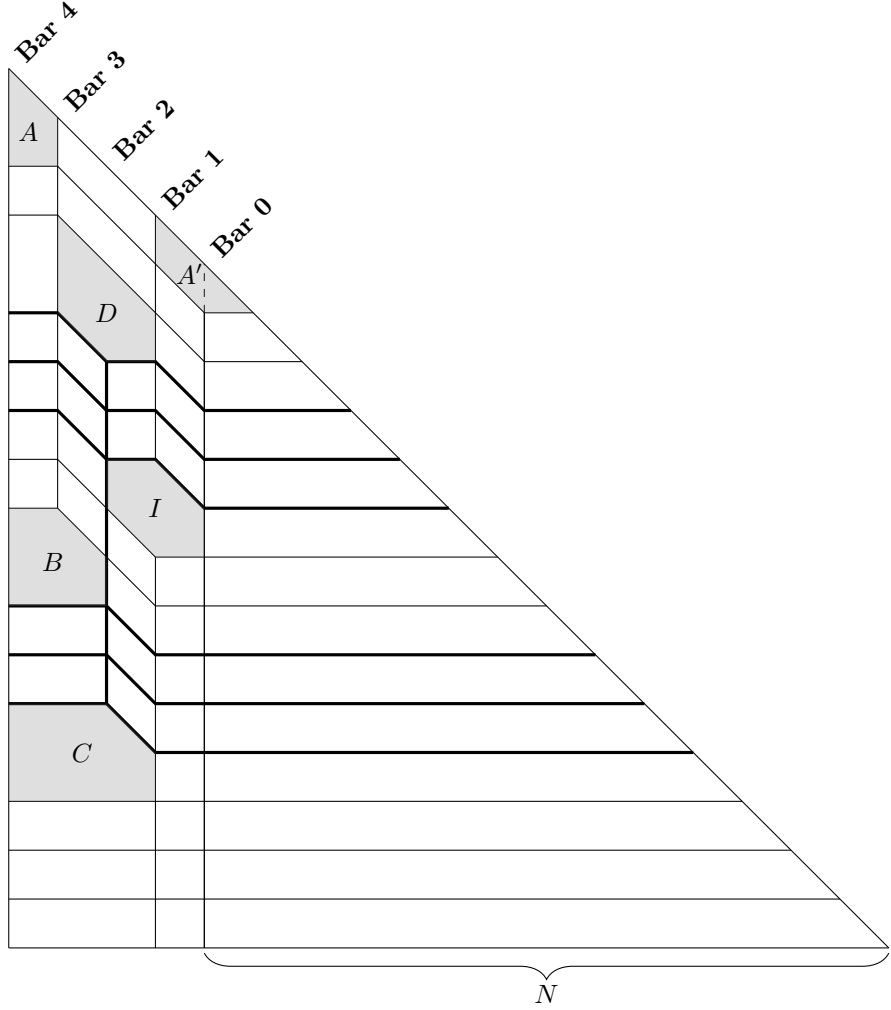
\begin{figure}
\begin{tikzpicture}[scale=.65]
\draw (0,0) -- (0,18);
\draw (3,0) -- (3,8);
\draw (4,0) -- (4,13);
\draw[very thick] (2,5) -- (2,12);
\draw (1,9) -- (1,17);
\draw (3,10) -- (3,15);
\draw (4,0) -- (4,13);
\draw[dashed] (4,13) -- (4,14);

\draw (0,0) -- (18,0);
\draw (0,1) -- (17,1);
\draw (0,2) -- (16,2);
\draw (0,3) -- (15,3);
\draw[very thick] (0,5) -- (2,5) -- (3,4) -- (14,4);
\draw[very thick] (0,6) -- (2,6) -- (3,5) -- (13,5);
\draw[very thick] (0,7) -- (2,7) -- (3,6) -- (12,6);
\draw (0,9) -- (1,9) -- (3,7) -- (11,7);
\draw (0,10) -- (1,10) -- (3,8) -- (10,8);
\draw[very thick] (0,11) -- (1,11) -- (2,10) -- (3,10) -- (4,9) -- (9,9);
\draw[very thick] (0,12) -- (1,12) -- (2,11) -- (3,11) -- (4,10) -- (8,10);
\draw[very thick] (0,13) -- (1,13) -- (2,12) -- (3,12) -- (4,11) -- (7,11);
\draw (0,15) -- (1,15) -- (4,12) -- (6,12);
\draw (0,16) -- (1,16) -- (4,13) -- (5,13);

\draw (0,18) -- (18,0);

\draw [decorate,decoration={brace,amplitude=10pt},yshift=-3pt]
(18,0) -- (4,0) node [black,midway,yshift=-0.55cm] 
{$N$};
\fill[gray,opacity=.25] (0,3) -- (0,5) -- (2,5) -- (3,4) -- (3,3) -- cycle;
\fill[gray,opacity=.25] (0,7) -- (0,9) -- (1,9) -- (2,8) -- (2,7) -- cycle;
\fill[gray,opacity=.25] (0,16) -- (0,18) -- (1,17) -- (1,16) -- cycle;
\fill[gray,opacity=.25] (2,9) -- (2,10) -- (3,10) -- (4,9) -- (4,8) -- (3,8) -- cycle;
\fill[gray,opacity=.25] (3,14) -- (3,15) -- (5,13) -- (4,13) -- cycle;
\fill[gray,opacity=.25] (3,12) -- (2,12) -- (1,13) -- (1,15) -- (3,13) -- cycle;

\node[black] at (1.5,4){$C$};
\node[black] at (.9,7.9){$B$};
\node[black] at (.4,16.7){$A$};
\node[black] at (3,9){$I$};
\node[black] at (3.7,13.8){$A'$};
\node[black] at (2,13){$D$};

\node[black,rotate=45,xshift=.7cm] at (4,14){\bf Bar 0};
\node[black,rotate=45,xshift=.7cm] at (3,15){\bf Bar 1};
\node[black,rotate=45,xshift=.7cm] at (2,16){\bf Bar 2};
\node[black,rotate=45,xshift=.7cm] at (1,17){\bf Bar 3};
\node[black,rotate=45,xshift=.7cm] at (0,18){\bf Bar 4};

\end{tikzpicture}
\caption{An example of a construction from Lemma \ref{newtonConstruction} with the tuple $(a',b',a,b)=(2,5,7,10)$. The horizontal edges representing the two intervals between which equality is forced are shown in bold as is a vertical line representing bar 2.}
\label{newtonianPicture}
\end{figure}
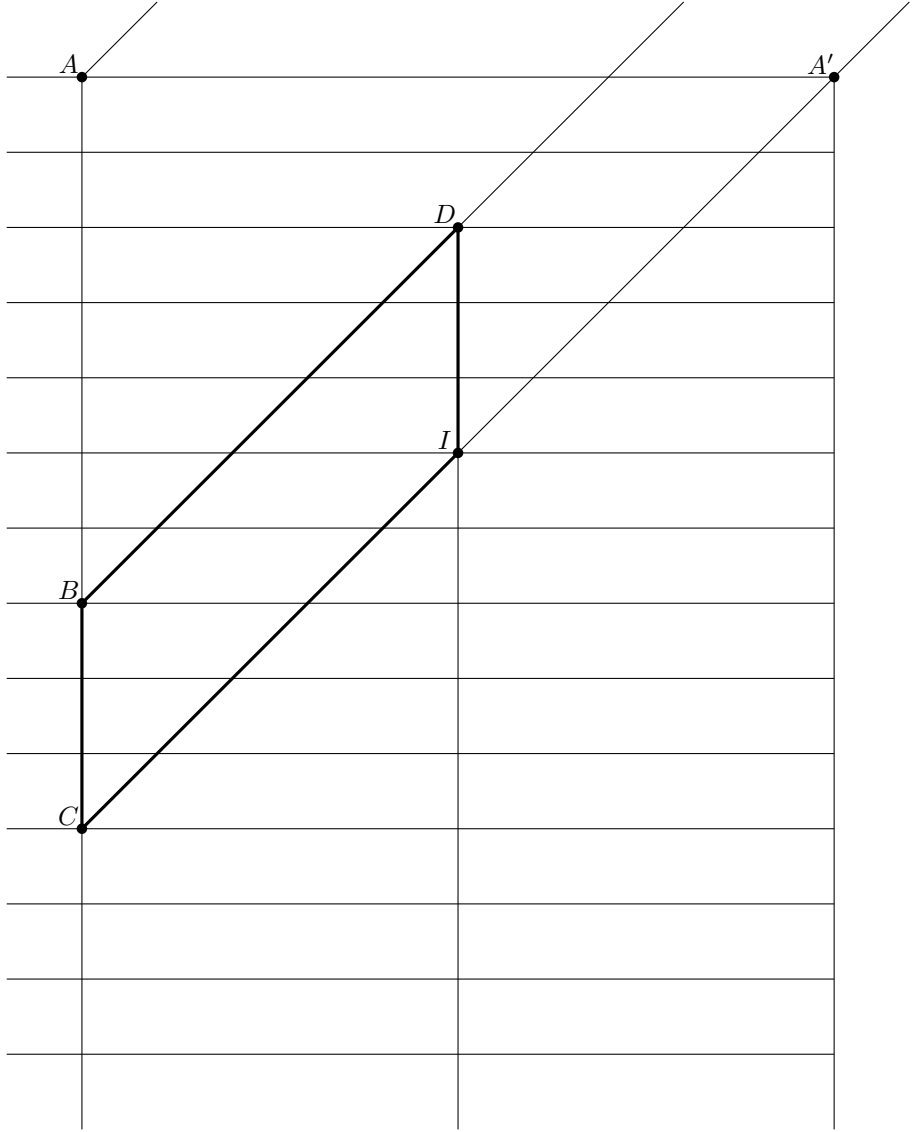
\begin{figure}
\begin{tikzpicture}
\draw (10,13) -- (10,-1);
\draw (-1,0) -- (10,0);
\draw (-1,1) -- (10,1);
\draw (-1,2) -- (10,2);
\draw (-1,3) -- (10,3);
\draw (-1,4) -- (10,4);
\draw (-1,5) -- (10,5);
\draw (-1,6) -- (10,6);
\draw (-1,7) -- (10,7);
\draw (-1,8) -- (10,8);
\draw (-1,9) -- (10,9);
\draw (-1,10) -- (10,10);
\draw (-1,11) -- (10,11);
\draw (-1,12) -- (10,12);
\draw (-1,13) -- (10,13);

\draw (0,13) -- (0,-1);
\draw (0,13) -- (1,14);
\draw (0,3) -- (11,14);
\draw (0,6) -- (8,14);
\draw (5,11) -- (5,-1);

\draw[very thick] (0,3) -- (0,6) -- (5,11) -- (5,8) -- cycle;

\fill[black] (0,3) circle (.7mm) node[black,xshift=-5pt,yshift=5pt]{$C$};
\fill[black] (0,6) circle (.7mm) node[black,xshift=-5pt,yshift=5pt]{$B$};
\fill[black] (0,13) circle (.7mm) node[black,xshift=-5pt,yshift=5pt]{$A$};
\fill[black] (5,8) circle (.7mm) node[black,xshift=-5pt,yshift=5pt]{$I$};
\fill[black] (5,11) circle (.7mm) node[black,xshift=-5pt,yshift=5pt]{$D$};
\fill[black] (10,13) circle (.7mm) node[black,xshift=-5pt,yshift=5pt]{$A'$};

\end{tikzpicture}
\caption{The arrangement represented by the leftmost four columns of $N'$. Everything to the right of the rightmost line is determined by $N$. The parallelogram in bold is the main focus of the construction, as the length of opposite sides must be equal.}
\label{nonnewtonianPicture}
\end{figure}

\textbf{Construction of $N'$ from $N$.} Let $v_{-1},\ldots,v_m$ be the vertices of $N$ on the left edge. We will place additional faces along this left edge to create a larger subdivision $N'$. We now divide our construction into a number of vertical components which we will call \textbf{bars}, which correspond to contiguous vertical segments in $N'\setminus N$. We do not use these as formal objects, but rather as guides for the reader. Bar 0 will correspond to the left edge of $N$. We begin by defining a set of vertices we will use to extend $N$.
\begin{description}
\item[Bar 0 to 1a] If $-1\leq i < b'$, define $v'_{1,i}=v_i+\bar e_1$.
\item[Bar 0 to 1b] If $b'\leq i \leq m$, define $v'_{1,i}=v_i-\bar e_3$.
\item[Bar 1a to 3] If $-1\leq i < a'$, define $v'_{3,i}=v'_{1,i}+2\bar e_1$.
\item[Bar 1a to 2] If $a'\leq i < b'$, define $v'_{2,i}=v'_{1,i}-\bar e_3$.
\item[Bar 1b to 2] If $b'\leq i < b$, define $v'_{2,i}=v'_{1,i}+\bar e_1$.
\item[Bar 1b to 4] If $b\leq i \leq m$, define $v'_{4,i}=v'_{1,i}-3\bar e_3$.
\item[Bar 2 to 3] If $a'\leq i < a$, define $v'_{3,i}=v'_{2,i}+\bar e_1$.
\item[Bar 2 to 4] If $a\leq i <b$, define $v'_{4,i}=v'_{2,i}-2\bar e_3$.
\item[Bar 3 to 4] If $i=-1$, define $v'_{4,i}=v'_{3,i}+\bar e_1$. If $0\leq i <a$, define $v'_{4,i}=v'_{3,i}-\bar e_3$.
\end{description}
We will define the vertices of $N'$ to be the vertices of $N$ along with all the additional $v'_{k,i}$ except for $v_{-1}$.

Using these vertices, we begin to specify edges of $N'$. In particular, we define the edges of $N'$ to be the edges in the following seven cases:
\begin{description}
\item[Edges from $N$]: Let $N'$ contain every edge of $N$ that does not contain $v_{-1}$. Additionally, let $v'_{-1,-1}$ be the vertex in $N$ that shares an edge in the direction of $\bar e_1$ with $v_{-1}$. Let $N'$ contain an edge connecting $v_{-1}$ to $v'_{1,-1}$.
\item[Non-vertical edges]: Whenever, in the above nine cases defining vertices, the value of $v'_{k',i}$ was defined from $v'_{k,i}$, there is an edge connecting $v'_{k,i}$ and $v'_{k',i}$. 
\item[Bar 0] There are edges $v_i$ to $v_{i+1}$ for $0\leq i < m$.
\item[Bar 1a] There are edges $v'_{1,i-1}$ to $v'_{1,i}$ for $0\leq i < b'$.
\item[Bar 1b] There are edges $v'_{1,i}$ to $v'_{1,i+1}$ for $b'\leq i < m$.
\item[Bar 2] There are edges $v_{2,i}$ to $v_{2,i+1}$ for $a'\leq i < b'-1$.
\item[Bar 3] There are edges $v_{3,i-1}$ to $v_{3,i}$ for $0\leq i < a$.
\item[Bar 4] There are edges $v_{4,i-1}$ to $v_{4,i}$ for $0\leq i < m$.
\end{description}
By computation, one may verify that the specified edges are all vertical and oriented in the proper direction.

We claim that the defined set of edges is the set of edges of some linear Newton subdivision. We do this by demanding that the faces of $N'$ be precisely the regions bounded by the given edges and checking that this defines a linear Newton subdivision. Most of the regions defined as such will be parallelograms as the construction gives many parallel segments between bars. The only faces of $N'$ outside of $N$ that are not parallelograms are at the ends of bars. 

Every face determined by the edges must be convex because positive combinations of the edges leaving a fixed interior vertex always span the plane, hence determine only angles less than $\pi$. As the edges are always parallel to $\bar e_1,\,\bar e_2$ or $\bar e_3$, we may test that the faces are of the form $P_{c,w_1,w_2,w_3}$ by observing that, when their boundary is traversed counterclockwise, the edge in the direction $\bar e_1$ is longer than the edge in the direction $-\bar e_1$. We may easily enumerate the polygons in $N'$ that are not parallelograms and check that they are of the desired form:
\begin{itemize}
\item There is some polygon $A'$ containing the vertices $v_0,\,v'_{1,0}$ and $v'_{1,-1}$ along with some vertices of $N$. Letting $f\in N$ be the face containing $v_{-1}$, we find that $A$ is the Minkowski sum of $f$ with a segment in the direction of $\bar e_1$, hence is of the form $P_{c,w_1,w_2,w_3}$.
\item There is a polygon $A$ with vertices $\{v'_{4,-1},v'_{3,-1},v'_{3,0},v'_{4,0}\}$.
\item There is a polygon $B$ with vertices $\{v'_{4,a-1},v'_{4,a},v'_{2,a},v'_{2,a-1},v'_{3,a-1}\}$
\item There is a polygon $C$ with vertices $\{v'_{4,b-1},v'_{4,b},v'_{1,b},v'_{1,b-1},v'_{2,b-1}\}$
\item There is a polygon $D$ with vertices $\{v'_{3,a'-1},v'_{3,a'},v'_{2,a'},v'_{1,a'},v'_{1,a'-1}\}$
\item There is a polygon $I$ with vertices $\{v'_{2,b'-1},v'_{2,b'},v'_{1,b'},v'_{0,b'},v'_{0,b'-1},v'_{1,b'-1}\}$
\end{itemize}

After verifying each of these special cases, we have that $N'$ is a linear Newton subdivision. This subdivision is $m$-extensible as every edge aside from the topmost one enters bar 4 horizontally and there are $m+2$ such edges. The bounded edges and vertices of $N$ are a subset of those of $N'$. Therefore, we can restrict an element $d\in \mathcal R(N')$ to the bounded edges of $N$ to give an element of $\mathcal R(N)$. Call this map $\sigma:\mathcal R(N')\rightarrow \mathcal R(N)$.

\textbf{Compatibility of $\sigma$ with $\pi_N$ and $\pi_{N'}$.} The construction of $N'$ has the property that for any vertex $v\in V_{N'}\setminus V_N$, there is exactly one edge $\ell_v$ containing $v$ such that $\overrightarrow{e_v}\in \{\bar e_1,-\bar e_3\}$ and exactly one $r_v$ such that $\overrightarrow{e_v}\in \{-\bar e_1,\bar e_3\}$. The possible arrangements of edges around each vertex, labeled with the possible values of $d$, are shown in Figure \ref{i'm a figure}. By the projection $\mathbb R^2\rightarrow \mathbb R^1$ killing $\bar e_2$, we see that any $d\in \mathcal R(N')$ must have $d(\ell_v)=d(r_v)$. This is because the construction creates various horizontal chains along which any $d\in\mathcal R(N')$ is constant. We then see that, if we have $v=v'_{i,k}$ for some $i,k$, we must have that $d(r_v)=d(\ell_v)=d(v_{i})=\pi_N(\sigma(d))_i$. In particular, we find that $\pi_{N'}=\pi_N\circ \sigma$.
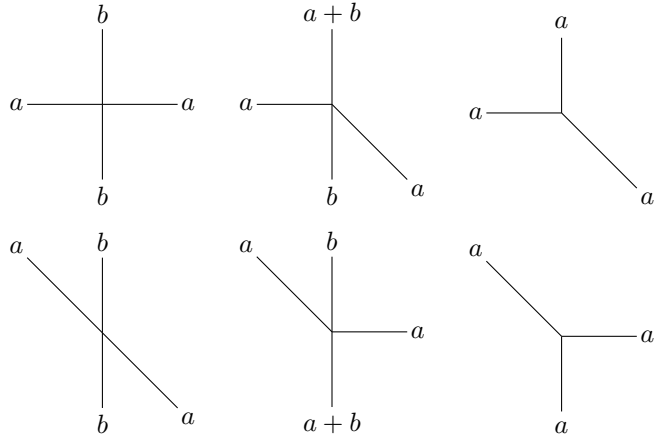
\begin{figure}
\begin{center}
\begin{tabular}{ccc}
\begin{tikzpicture}
\draw (0,0) -- (1,0) node[xshift=4pt]{$a$};
\draw (0,0) -- (-1,0) node[xshift=-4pt]{$a$};
\draw (0,0) -- (0,1) node[yshift=6pt]{$b$};
\draw (0,0) -- (0,-1) node[yshift=-6pt]{$b$};
\end{tikzpicture}
&
\begin{tikzpicture}
\draw (0,0) -- (1,-1) node[xshift=4pt,yshift=-4pt]{$a$};
\draw (0,0) -- (-1,0) node[xshift=-4pt]{$a$};
\draw (0,0) -- (0,1) node[yshift=6pt]{$a+b$};
\draw (0,0) -- (0,-1) node[yshift=-6pt]{$b$};
\end{tikzpicture}
&
\begin{tikzpicture}
\draw (0,0) -- (1,-1) node[xshift=4pt,yshift=-4pt]{$a$};
\draw (0,0) -- (-1,0) node[xshift=-4pt]{$a$};
\draw (0,0) -- (0,1) node[yshift=6pt]{$a$};
\end{tikzpicture}
\\
\begin{tikzpicture}
\draw (0,0) -- (1,-1) node[xshift=4pt,yshift=-4pt]{$a$};
\draw (0,0) -- (-1,1) node[xshift=-4pt,yshift=4pt]{$a$};
\draw (0,0) -- (0,1) node[yshift=6pt]{$b$};
\draw (0,0) -- (0,-1) node[yshift=-6pt]{$b$};
\end{tikzpicture}
&
\begin{tikzpicture}
\draw (0,0) -- (1,0) node[xshift=4pt]{$a$};
\draw (0,0) -- (-1,1) node[xshift=-4pt,yshift=4pt]{$a$};
\draw (0,0) -- (0,1) node[yshift=6pt]{$b$};
\draw (0,0) -- (0,-1) node[yshift=-6pt]{$a+b$};
\end{tikzpicture}
&
\begin{tikzpicture}
\draw (0,0) -- (1,0) node[xshift=4pt]{$a$};
\draw (0,0) -- (-1,1) node[xshift=-4pt,yshift=4pt]{$a$};
\draw (0,0) -- (0,-1) node[yshift=-6pt]{$a$};
\end{tikzpicture}

\end{tabular}
\end{center}
\caption{The six arrangements of edges appearing in the added portion of $N'$ labeled with a generic exact metric.}
\label{i'm a figure}
\end{figure}

\textbf{Injectivity of $\sigma$.} We have shown that the values of $d\in \mathcal R(N')$ on the non-vertical edges is determined by the value $\sigma(d)$. To conclude that $\sigma$ is injective, we must show that the values of $d$ on the vertical edges are also determined by $\sigma(d)$. We use the exactness condition to accomplish this.

Define a constant $\alpha_v$ which is $1$ if $\ell_v$ is in the direction $\bar e_1$ and $0$ otherwise and a constant $\beta_v$ which is $1$ if $r_v$ is in the direction $\bar e_1$ and $0$ otherwise. Let $\kappa_v=\alpha_v-\beta_v$. Note that $\kappa_v\in \{1,0,-1\}$ for all $v\in V_{N'}\setminus V_N$.

A vertex $v$ may have at most two vertical edges. We will call the upwards and downwards edges $u_v$ and $d_v$ respectively when they exist. Formally, we will write $d(u_v)=0$ or $d(d_v)=0$ if there is no upwards or downwards edge respectively. With this formalism, we find that $d(d_v)=d(u_v)+\kappa_v d(\ell_v)$, by the exactness condition. Letting $v=v'_{i,k}$, we get that $d(d_v)=d(u_v)+\kappa_v \pi_N(\sigma(d))_i$. 
Given $\sigma(d)$, this equation allows us to determine $d(d_v)$ from $d(u_v)$ and vice versa. In the special case where a vertex has only one vertical edge and either $d(d_v)$ or $d(u_v)$ is taken to be zero, the equation yields the value of $d$ on the existing vertical edge. The vertical segments bar 0, 1a, 1b, 2, and 3 all contain some vertex with only one vertical line and are contiguous, hence the values of $d$ on the vertical edges are determined by $\sigma(d)$. Thus, $\sigma$ is injective.

\textbf{The image of $\sigma$.} Fix some $d'\in\mathcal R(N)$. Let us derive a condition equivalent to the existence of some $d\in\mathcal R(N')$ such that $\sigma(d)=d'$. 

Bar 0 and bar 1b have $\kappa_v\in \{1,0\}$ everywhere with a single degree-three vertex at the top. Thus the values of any such $d'$ on this bar can be written as sums of positive quantities and are themselves positive. Bar 1a and bar 3 have that $\kappa_v\in \{0,-1\}$ everywhere with a single vertex at the bottom with degree $3$. The values of $d'$ can again be written as sums of positive quantities by working bottom to top. So these bars never pose any obstacle to extending $d'$ to $N'$.

The only remaining question is when the values of $d$ can be chosen suitably on bar 2. Because the value of $d$ on the non-vertical edges is determined by $d'=\sigma(d)$, the exactness condition yields a system of equations giving the value of $d$ on any vertical edge in terms of the value of $d$ on the edge above. The exactness condition also determines the value of $d$ on the highest and lowest edges. This system is overdetermined and the existence of a solution $d$ reduces to the following equation on $d'$.

$$\sum_{a'\leq i < b}\kappa_{v_{2,i}}.\pi_{N}(d')_i=0$$

Expanding $\kappa_v=\alpha_v-\beta_v$ and noting that for $a'\leq i <b$ we have that $\alpha_{v_{2,i}}=1$ if and only if $a'\leq i<a$ and $\beta_{v_{2,i}}=1$ if and only if $b'\leq i< b$, we find that
$$\sum_{a'\leq i < a}\pi_{N}(d')_i=\sum_{b'\leq i < b}\pi_N(d')_i.$$

By subtracting $\sum_{b'\leq i < a}\pi_N(d')_i$ from each side, we get
\begin{equation}\sum_{a'\leq i < b'}\pi_{N}(d')_i=\sum_{a\leq i < b}\pi_{N}(d')_i.\label{theconditionything}\end{equation}

As long as this condition is satisfied, we may choose values for $d$ on bar 2 that satisfy the exactness condition. The fact that these values are positive is trivial given that $\kappa_{v_{2,i}}$ is non-increasing in $i$. Thus, an exact metric $d'$ on $N$ extends uniquely an exact metric $d$ on $N'$ if and only if Equation \ref{theconditionything} is satisfied. Equivalently, the image of $\sigma$ is precisely the subspace of $\mathcal R(N)$ on which the the given equation holds.
\end{proof}
We can define a Newton subdivision $N_m$ of a triangle with side length $m+2$ such that its edges are exactly horizontal edges from one side of the subdivision to the other and every face is of the form $P_{1,0,0,w_3}$. This is shown in the right half of Figure \ref{newtonianPicture}. This construction is the Newton subdivision of a series of tropical lines sharing a common vertical ray. In addition, $N_m$ is $m$-extensible and the induced map $\pi_{N_m}:\mathcal R(N_m)\rightarrow\mathbb R_{>0}^m$ is an isomorphism, as all the interior edges in $N_m$ touch the left edge of the subdivision. There are no constraints on the values of an exact metric on $N_m$ since there are no interior vertices.

A straightforward induction using the previous lemma shows the following lemma.
\begin{lemma}
For any intervallic subspace $S$ of $\mathbb R_{>0}^m$, there exists an extensible Newton subdivision $N$ such that $\mathcal R(N)$ is isomorphic to $S$ by $\pi_N$.
\end{lemma}

In Figure \ref{complete}, we diagram an arrangement of tropical lines.

\begin{figure}
\begin{center}
\begin{tikzpicture}
\draw[black] (0,0) -- (-13,0);
\draw[black] (0,2) -- (-13,2);
\draw[black] (0,2.5) -- (-13,2.5);
\draw[black] (0,3.5) -- (-13,3.5);
\draw[black] (0,4) -- (-13,4);
\draw[black] (0,4.5) -- (-13,4.5);
\draw[black] (0,5.5) -- (-13,5.5);

\draw[black] (0,5.5) -- (0,-2);

\draw[black] (0,0) -- (2,2);
\draw[black] (0,2) -- (2,4);
\draw[black] (0,2.5) -- (2,4.5);
\draw[black] (0,3.5) -- (2,5.5);
\draw[black] (0,4) -- (2,6);
\draw[black] (0,4.5) -- (2,6.5);

\draw[black] (-3.5,2) -- (2,7.5);
\draw[black] (-3.5,4) -- (0,7.5);
\draw[black] (-3.5,5.5) -- (-3.5,-2);
\draw[black] (-2,5.5) -- (-2,-2);

\draw[black] (-9,0) -- (-1.5,7.5);
\draw[black] (-9,2) -- (-3.5,7.5);
\draw[black] (-9,5.5) -- (-9,-2);
\draw[black] (-7,4) -- (-7,-2);

\draw[black] (-7,7.5) -- (-11,3.5);
\draw[black] (-7.5,7.5) -- (-11,4);
\draw[black] (-11,5.5) -- (-11,-2);
\draw[black] (-10.5,4.5) -- (-10.5,-2);
\draw[black] (-11,5.5) -- (-9,7.5);

\fill[black] (0,0) circle (2pt);
\fill[black] (0,2) circle (2pt);
\fill[black] (0,2.5) circle (2pt);
\fill[black] (0,3.5) circle (2pt);
\fill[black] (0,4) circle (2pt);
\fill[black] (0,4.5) circle (2pt);
\fill[black] (0,5.5) circle (2pt);

\fill[black] (-3.5,2) circle (2pt);
\fill[black] (-3.5,4) circle (2pt);
\fill[black] (-3.5,5.5) circle (2pt);
\fill[black] (-2,5.5) circle (2pt);

\fill[black] (-9,0) circle (2pt);
\fill[black] (-9,2) circle (2pt);
\fill[black] (-9,5.5) circle (2pt);
\fill[black] (-7,4) circle (2pt);

\fill[black] (-11,3.5) circle (2pt);
\fill[black] (-11,4) circle (2pt);
\fill[black] (-11,5.5) circle (2pt);

\fill[black] (-10.5,4.5) circle (2pt);

\draw [<->] (-12,0.1) -- (-12,1.9) node [black,midway,xshift=-.3cm] {$x$};
\draw [<->] (-12,2.1) -- (-12,2.4) node [black,midway,xshift=-.9cm] {$x-y-z$};
\draw [<->] (-12,2.6) -- (-12,3.4) node [black,midway,xshift=-.3cm] {$y$};
\draw [<->] (-12,3.6) -- (-12,3.9) node [black,midway,xshift=-.3cm] {$z$};
\draw [<->] (-12,4.1) -- (-12,4.4) node [black,midway,xshift=-.3cm] {$z$};
\draw [<->] (-12,4.6) -- (-12,5.4) node [black,midway,xshift=-.6cm] {$x-2z$};

\draw [decorate,decoration={brace},yshift=-.1cm] (0.1,-2) -- (-3.4,-2) node [black,midway,yshift=-.5cm] 
{Constraint I};
\draw [decorate,decoration={brace},yshift=-.1cm] (-3.6,-2) -- (-8.9,-2) node [black,midway,yshift=-.5cm] 
{Constraint II};
\draw [decorate,decoration={brace},yshift=-.1cm] (-9.1,-2) -- (-10.9,-2) node [black,midway,yshift=-.5cm] 
{Constraint III};

\end{tikzpicture}
\end{center}
\caption{A construction showing an arrangement with $3$ marked distances $x,\,y,\,z$ enforcing the constraints that $x>y+z$ and $x>2z$. Each of the constraints labeled indicate that the corresponding distances cannot be deformed freely, but must satisfy a linear inequality.}
\label{complete}
\end{figure}
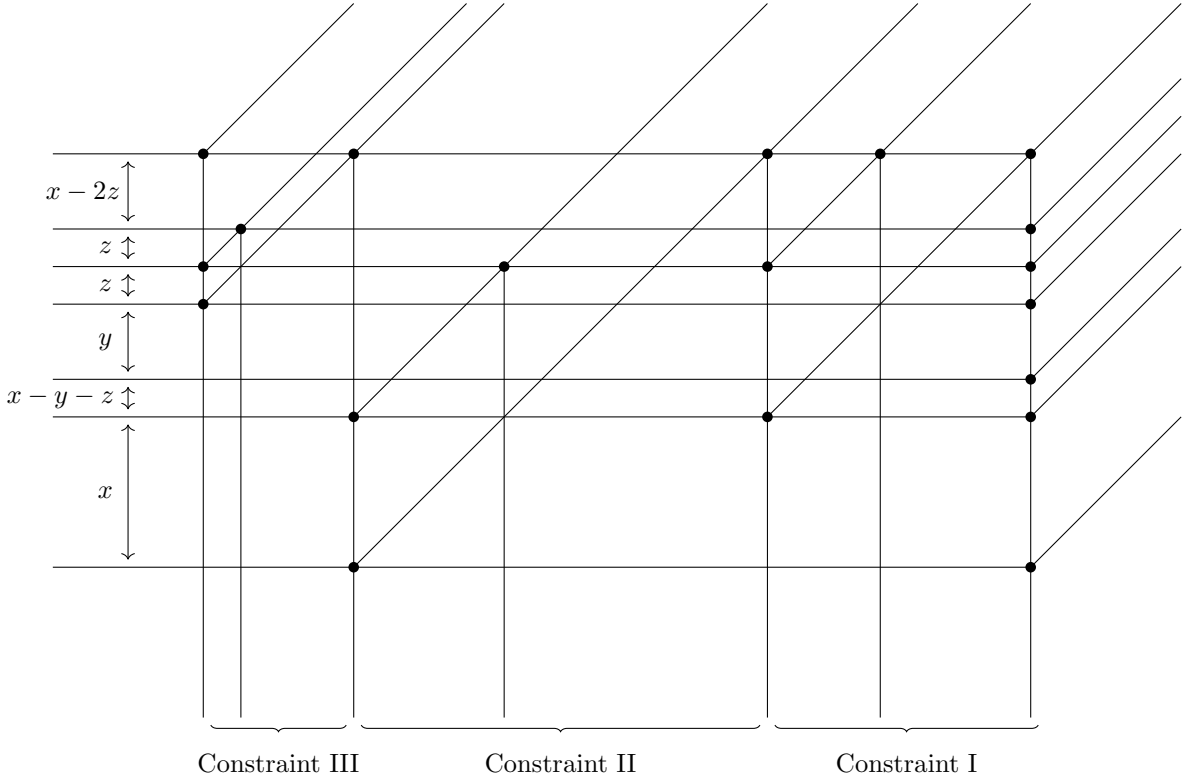
\subsection{Subspaces of the positive orthant}
We show that we can construct intervallic subspaces linearly isomorphic to any subspace of $\mathbb R_{>0}^m$ defined by linear equalities and inequalities. Define the canonical projection $\pi_{m',m}:\mathbb R^{m'}\rightarrow\mathbb R^m$ for $m'\geq m$ taking a vector to its first $m$ coordinates.

\begin{lemma}Fix some intervallic subspace $S\subseteq \mathbb R_{>0}^m$ and any integers $0\leq a,b<m$. There exists an intervallic subspace $S'\subseteq \mathbb R_{>0}^{m+3}$ such that the canonical projection $\pi_{m+3,m}$ is a linear isomorphism $S'\rightarrow S$ and for every $v\in S'$ we have $v_{m+2}=v_a+v_b$.
\label{addition}
\end{lemma}
\begin{proof}Let $I$ be a set of constraints determining the set $S$ and define \[I'=I\cup \{(m,m+1,a,a+1),(m+1,m+2,b,b+1),(m+2,m+3,m,m+2)\}.\]
Let $S'$ be the subspace determined by $I'$. Define a map $f:\mathbb R^m\rightarrow\mathbb R^{m+3}$ by \[f(v_0,\ldots,v_{m-1})=(v_0,\ldots,v_{m-1},v_a,v_b,v_a+v_b).\] Note that $f$ takes $S$ to $S'$. More strongly, the restriction of $f$ to $S$ is the inverse of the restriction of the $\pi_{m',m}$ to $S'$.\end{proof}

By repeatedly applying this observation, we establish the following lemma. 

\begin{lemma}
Let $S\subseteq \mathbb R_{>0}^m$ be an intervallic subspace and fix some set $F$ of non-zero linear functions $f:\mathbb R_{>0}^m\to \mathbb R_{>0}$ with non-negative integer coefficients. There exists a $m'\geq m$ and an intervallic subspace $S'\subseteq \mathbb R_{>0}^{m'}$ such that the projection $\pi_{m',m}$ is a linear isomorphism $S'\rightarrow S$ and such that for every $f\in F$, there exists some $0\leq i < m'$ such that for all $v\in S'$ we have $v_i=f(\pi_{m',m}(v))$. 
\label{positiveOrthantUniversality}
\end{lemma}
\begin{proof}
We proceed by induction on the sum of the coefficients of every $f\in F$. Choose some $f\in F$. We handle two cases. If the sum of the coefficients in $f$ is $1$, then $f(v)=v_a$ for some $0\leq a <m$. By the inductive hypothesis, choose some intervallic subspace $S'$ satisfying the conditions of the lemma for the set $F\setminus \{f\}$. Then, $S'$ satisfies the conditions of the lemma for the set of functions $F$ as well since $v_a=f(\pi_{m',m}(v))$ by definition of $f$.

If the sum of the coefficients in $f$ is greater than $1$, write $f(v)=v_a+f'(v)$ for some $0\leq a< m$. By induction, choose an intervallic set $S'\subseteq \mathbb R^{m'}$ satisfying the conditions of the lemma for the set of functions $F\setminus \{f\}\cup \{f'\}$. Let $0\leq b<m'$ be such that $v_b=f'(\pi_{m',m}(v))$. Now, apply Lemma \ref{addition} to get a intervallic set $S''\subseteq \mathbb R^{m'+3}$ such that $v_{m'+3}=v_a+v_b=v_a+f'(v)=f(v)$ and $\pi_{m+3,m}$ is a linear isomorphism $S''\rightarrow S'$. Observe that for any $g\in F\setminus \{f\}$ we have that $g\in\setminus\{f\}\cup\{f'\}$, thus there is some $i$ such that $v_i=f(\pi_{m',m}(v))$ for any $v\in S'$ and, since $\pi_{m',m}\circ\pi_{m'+3,m'}=\pi_{m'+3,m}$, we have $v_i=f(\pi_{m'+3,m}(v))$ for any $v\in S''$. Thus, $S''$ satisfies the conditions of the lemma.
\end{proof}

This lemma tells us enough about the structure of intervallic sets to prove the following lemma necessary to the proof of universality.
\begin{lemma}
Let $E_0$ and $E_+$ be two sets of linear functions on $\mathbb R_{>0}^m$ with integer coefficients. Define a subspace 
\[V=\{v\in\mathbb R_{>0}^m:\forall f\in E_0[f(v)=0]\text{ and }\forall f\in E_+[f(v)>0]\}.\]
Then $V$ is linearly isomorphic to an intervallic subspace.
\label{ppositive}
\end{lemma}
\begin{proof}
For any linear function $f:\mathbb R_{>0}^m\rightarrow\mathbb R$ with coefficients in $\mathbb Z$, define functions $f^+$ and $f^-$ to be the terms with positive coefficients in $f$ and $-f$ respectively. Observe that $f=f^+-f^-$. The condition that $f(v)=0$ is equivalent to $f^+(v)=f^{-}(v)$. The condition $f(v)>0$ is equivalent to the existence of some unique positive $x$ such that $f^+(v)=f^{-}(v)+x$. Our proof strategy will be to express everything in terms of the equality of linear functions with non-negative integer coefficients, introducing as many slack variables as necessarily.

Let $m'=m+|E_+|$. Fix some bijection $s:E_+\rightarrow \{m,\ldots,m+|E_+|-1\}$. For each $f\in E_+$ define a function $e_{f}:\mathbb R_{>0}^{m'}\rightarrow \mathbb R_{>0}$ which takes $(v_0,\ldots,v_{m'-1})$ to $v_{s(f)}$. For each $f\in E_0\cup E_+$, define $g_f^+=f^+\circ \pi_{m',m}$. If $f\in E_0$, define $g_f^-=f^-\circ \pi_{m',m}$. If $f\in E_+$, define $g_f^-=f^-\circ \pi_{m',m}+e_f$. These functions are modifications of the original set of constraints, changed to operate on $\mathbb R_{>0}^{m'}$ using the last $|E_+|$ coordinates as slack variables.

Let $S=\mathbb R_{>0}^{m'}$. This is an intervallic subspace. Define a set of linear functions $F$ by
\[
F=\{g_f^+:f\in E_0\cup E_+\}\cup \{g_f^-:f\in E_0\cup E_+\}.
\]
By Lemma \ref{positiveOrthantUniversality}, there is some intervallic subspace $S'\subseteq \mathbb R_{>0}^{m''}$ such that the canonical projection $\pi_{m'',m'}$ is a linear isomorphism $S'\rightarrow S$ and for every $f\in F$ there is some $i$ such that $v_i=f(\pi(v))$. Fix some set of constraints $I'$ that determines $S'$.

For each $f\in E_0\cup E_+$, choose a pair $(a,b)$ with $0\leq a,b < m''$ such that, for all $v\in S'$ we have that $v_a=g_f^+(\pi_{m'',m'}(v))$ and $v_b=g_f^-(\pi_{m'',m'}(v))$. Define an admissible tuple \[i_f=(\min(a,b),\min(a,b)+1,\max(a,b),\max(a,b)+1).\]
We augment our set of constraints $I'$ by adding these new tuples to the set of constraints:
$$I''=I'\cup \{i_f:f\in E_0\}\cup \{i_f:f\in E_+\}.$$
As a constraint, the tuple $i_f$ forces that $g_f^+(\pi_{m'',m'}(v))=g_f^-(\pi_{m'',m'}(v))$ on $S''$.

Let $S''$ be the set determined by $I''$. We may study this set using the fact that $\pi_{m'',m'}$ is a linear isomorphism from $S'$ to $S$. Let $\pi^{-1}:S\rightarrow S'$ be the inverse of $\pi_{m'',m'}$ restricted to $S'$. Note that $v\in \pi_{m'',m'}(S'')$ exactly if $\pi^{-1}(v)\in S''$. Since $\pi^{-1}(v)$ automatically satisfies every constraint in $I'$, it is only necessary to check that it satisfies the added constraints of the form $i_f$. This is equivalent to asking that $g_f^+(v)=g_f^-(v)$ for all $f\in E_0\cup E_+$.

Equivalently $\pi_{m'',m'}(S'')$ is precisely the subspace of $v\in S$ on which $f(\pi_{m',m}(v))=0$ for all $f\in E_0$ and $f(\pi_{m',m}(v))=v_{s(f)}$ for all $f\in E_+$. Given that the last $|E_+|$ coordinates of an element of $\pi_{m'',m'}(S'')$ are determined by the first $m$, we find that $\pi_{m',m}$ is injective on $\pi_{m'',m'}(S'')$. By composition, we find that $\pi_{m'',m}$ is also injective on $S''$. Finally, the image $\pi_{m'',m}(S'')$ is precisely the set of $v$ for which $f(v)=0$ for any $f\in E_0$ and $f(v)>0$ for any $f\in E_0$. This image equals $V$ and is linearly isomorphic to $S''$, which is intervallic.
 \end{proof}
 
In combination, Lemmas \ref{ppositive} and \ref{newtonConstruction} establish Theorem \ref{universalitythm}.

\bibliographystyle{siam}
\bibliography{bibliography}

\end{document}